\documentclass[12pt,letterpaper]{amsart}

\usepackage{palatino, euler, epic, eepic, xy, microtype, stmaryrd, epsfig, epstopdf,color}
\usepackage{graphicx,float}%,subfig}

\xyoption{all}

\usepackage{enumerate}
\usepackage[a4paper,bindingoffset=0.2in,left=0.8in,right=1.2in,top=1in,bottom=1in,footskip=.25in,marginparwidth=0.8in]{geometry}
\usepackage{float}
\usepackage{caption}
\usepackage{amsmath, amssymb, amsfonts, amsthm, accents, xcolor, enumerate, esint}
\usepackage{mathtools}
\usepackage{tikz-cd}
\usepackage{hyperref}
\usepackage{empheq}
\usepackage{adjustbox}

\usepackage{amsmath, amssymb, amscd, verbatim, xspace,amsthm}
\usepackage{latexsym, epsfig, color}

\renewcommand{\P}{\ensuremath{{\mathbb{P}}}}

\newcommand\isom{\cong}

\newcommand\GL{\operatorname{GL}}

\newcommand\bq{\begin{equation}}
  \newcommand\eq{\end{equation}}

\newtheorem{proposition}{Proposition}[section]
\newtheorem{theorem}[proposition]{Theorem}

\newtheorem{example}[proposition]{Example}

\newtheorem{lemma}[proposition]{Lemma}

\theoremstyle{definition}

\newtheorem*{nn-definition}{Definition}

\theoremstyle{remark}
\newtheorem{remark}[proposition]{Remark}
\usepackage{url}

\numberwithin{equation}{section}

\newcommand{\cut}[1]{}
\newcommand\hidden[1]{}

\newcommand{\M}{\mathbf{M}}

\newcommand{\K}{\widehat{\mathcal{M}_{\mathbf{L}}}}

\newcommand{\PP}{\mathbb{P}}
\newcommand{\QQ}{\mathbb{Q}}
\newcommand{\RR}{\mathbb{R}}

\newcommand{\ZZ}{{\mathbb{Z}}}                                        %
\newcommand{\Bl}{\operatorname{Bl}}                                   %                                                         %
\usepackage[makeroom]{cancel}                                                               %
\usepackage{caption}                                                  %
\usepackage{subcaption}                                               %
\usepackage{tikz-cd}                                                  %
\usetikzlibrary{calc}                                                 %
\usetikzlibrary{positioning}                                          %
\usetikzlibrary{decorations.pathreplacing}                            %
\usetikzlibrary{angles}                                               %
\usetikzlibrary{quotes}                                               %
\usetikzlibrary{shapes.misc}                                          %
\usetikzlibrary{decorations.text}                                     %
\usetikzlibrary{arrows}                                               %
\usetikzlibrary{spy}
\usetikzlibrary{intersections}
\usetikzlibrary{through}
\usetikzlibrary{backgrounds}

\usetikzlibrary{arrows.meta}

\usepackage{pgfplots}
\definecolor{pal1}{RGB}{215,48,39}
\definecolor{pal2}{RGB}{252,141,89}
\definecolor{pal3}{RGB}{254,224,144}
\definecolor{pal4}{RGB}{145,191,219}
\definecolor{pal5}{RGB}{69,117,180}
\pgfkeys{/pgf/number format/relative round mode=fixed}
\pgfplotsset{
  contour/label node code/.code={
    \node{\pgfmathprintnumber{#1}\,ms};
  }
}

\title{The geography of negative curves}

\author{Javier Gonz\'alez Anaya, Jos\'e Luis Gonz\'alez and Kalle Karu}
\address{
  J. Gonz\'alez-Anaya, Department of Mathematics, University of California, Riverside,
  Riverside, CA 92521, United States.  \newline \indent
  J.L. Gonz\'alez, Department of Mathematics, University of California, Riverside,
  Riverside, CA 92521, United States.  \newline \indent
  K. Karu,
  Department of Mathematics, University of British Columbia, 
  Vancouver, BC V6T1Z2, Canada.} 
\email{javiergo@ucr.edu, jose.gonzalez@ucr.edu, karu@math.ubc.ca}
\thanks{The second author was supported by a grant from the Simons Foundation (Award Number 710443) and by the UCR Academic Senate. The third author was supported by a NSERC Discovery grant.}

\begin{document}
\begin{abstract}
We study the Mori Dream Space (MDS) property for blowups of weighted projective planes at a general point and, more generally, blowups of toric surfaces defined by a rational plane triangle. The birational geometry of these varieties is largely governed by the existence of a negative curve in them, different from the exceptional curve of the blowup. 

We consider a parameter space of all rational triangles, and within this space we study how the negative curves and the MDS property vary. One goal of the article is to catalogue all known negative curves and show their location in the parameter space. In addition to the previously known examples we construct two new families of negative curves. One of them is, to our knowledge, the first infinite family of special negative curves. 

The second goal of the article is to show that the knowledge of negative curves in the parameter space often determines the MDS property.  We show that in many cases this is the only underlying mechanism responsible for the MDS property. 
\end{abstract}
\maketitle
\setcounter{tocdepth}{1} % this just includes subsections

% \tableofcontents

% ************************************************************************************************************************

%%%%%%%%%%%%%%%%%%%%%%%%%%%%%%%%%%%%%%%%%%%%%%%%%%%%%%%%%%%%%%%%%%%%%%%%%%%%%%%%%%%% 
%%%%%%%%%%%%%%%%%%%%%%%%%%%%%%%%%%%%%%%%%%%%%%%%%%%%%%%%%%%%%%%%%%%%%%%%%%%%%%%%%%%% 
%%%%                                                                            %%%% 
%%%% Section 1. Introduction                         %%%%
%%%%                                                                            %%%% 
%%%%%%%%%%%%%%%%%%%%%%%%%%%%%%%%%%%%%%%%%%%%%%%%%%%%%%%%%%%%%%%%%%%%%%%%%%%%%%%%%%%% 
%%%%%%%%%%%%%%%%%%%%%%%%%%%%%%%%%%%%%%%%%%%%%%%%%%%%%%%%%%%%%%%%%%%%%%%%%%%%%%%%%%%% 

\section{Introduction}

We work over an algebraically closed field $k$ of characteristic zero.

Let us consider the blowup of a weighted projective plane $\PP(a,b,c)$ at a general point, and more generally the blowup of a projective toric surface $X_\Delta$ with Picard number $1$, again at a general point. The main problems we are studying are: (i) find the Mori cone of curves of this variety, and (ii) determine if the variety is a Mori Dream Space (MDS). While previous work has concentrated on constructing examples and non-examples of MDS, in this article we start a more systematic study of families of such varieties. 

Let us describe the setup. Let $\Delta$ be a rational plane triangle, defining a projective toric surface $X_\Delta$ with Picard number $1$, and let $X$ be the blowup of this surface at a general point $e$,
\[X= \Bl_e X_\Delta.\]
A negative curve $C$ in $X$ is an irreducible curve of non-positive self-intersection, $C\cdot C\leq 0$, different from the exceptional curve $E$ of the blowup. The classes of $C$ and $E$ generate the Mori cone of curves of $X$. Hence, finding such a negative curve $C$ determines the Mori cone as well as the dual nef cone. Moreover, if $C$ has strictly negative self-intersection, $C\cdot C<0$, then this $C$ is unique in $X$.

Once a negative curve $C$ in $X$ is known, there is one more step to deciding if $X$ is a MDS. By a result of Cutkosky \cite{Cutkosky}, $X$ is a MDS if and only if it contains another curve $D$ such that $C\cap D=\emptyset$. The class of the curve $D$ and the hyperplane class $H$ in $X_\Delta$ generate the nef cone of $X$.

Our aim in this article is to study all triangles $\Delta$ and the pairs of non-intersecting curves $C$ and $D$ in them. Our main object of interest is not the negative curve $C\subset X$, but the restriction $C^\circ$ of this curve to the torus $T\subseteq X_\Delta$. Different varieties $X$ may contain negative curves with the same restrictions to the torus. More precisely, a negative curve $C$ in $X$ is defined by an irreducible Laurent polynomial $\xi(x,y)$ vanishing to order $m$ at $e$ and with support in a triangle $\Delta$ with area $\leq \frac{m^2}{2}$. For each such $\xi$ we may deform the triangle $\Delta$ and hence possibly get many other triangles for which $\xi$ defines a negative curve. All these negative curves are obtained from the same curve $C^\circ$ using the operations of closure and strict transform.

We consider the space of all rational triangles $\Delta$, and in this space we consider the subset of triangles that support a negative curve defined by some fixed polynomial $\xi$.
Two triangles that differ by scaling and translation define the same toric variety and hence should be considered the same; we call such triangles parallel.
A natural choice of parameters for this space are the three slopes of $\Delta$. Without loss of generality we will restrict our attention to a slice of this parameter space and assume that all our triangles $\Delta$ have a horizontal base, left slope $s$, right slope $t$ and $0<s<t$. Figure~\ref{fig-plot} shows the $st$-plane with many diamond shaped regions. Each such region corresponds to a polynomial $\xi$ and consists of all triangles $\Delta_{s,t}$ for which $\xi$ defines a negative curve. 

\begin{figure}
  \centering
  \includegraphics[width=\textwidth]{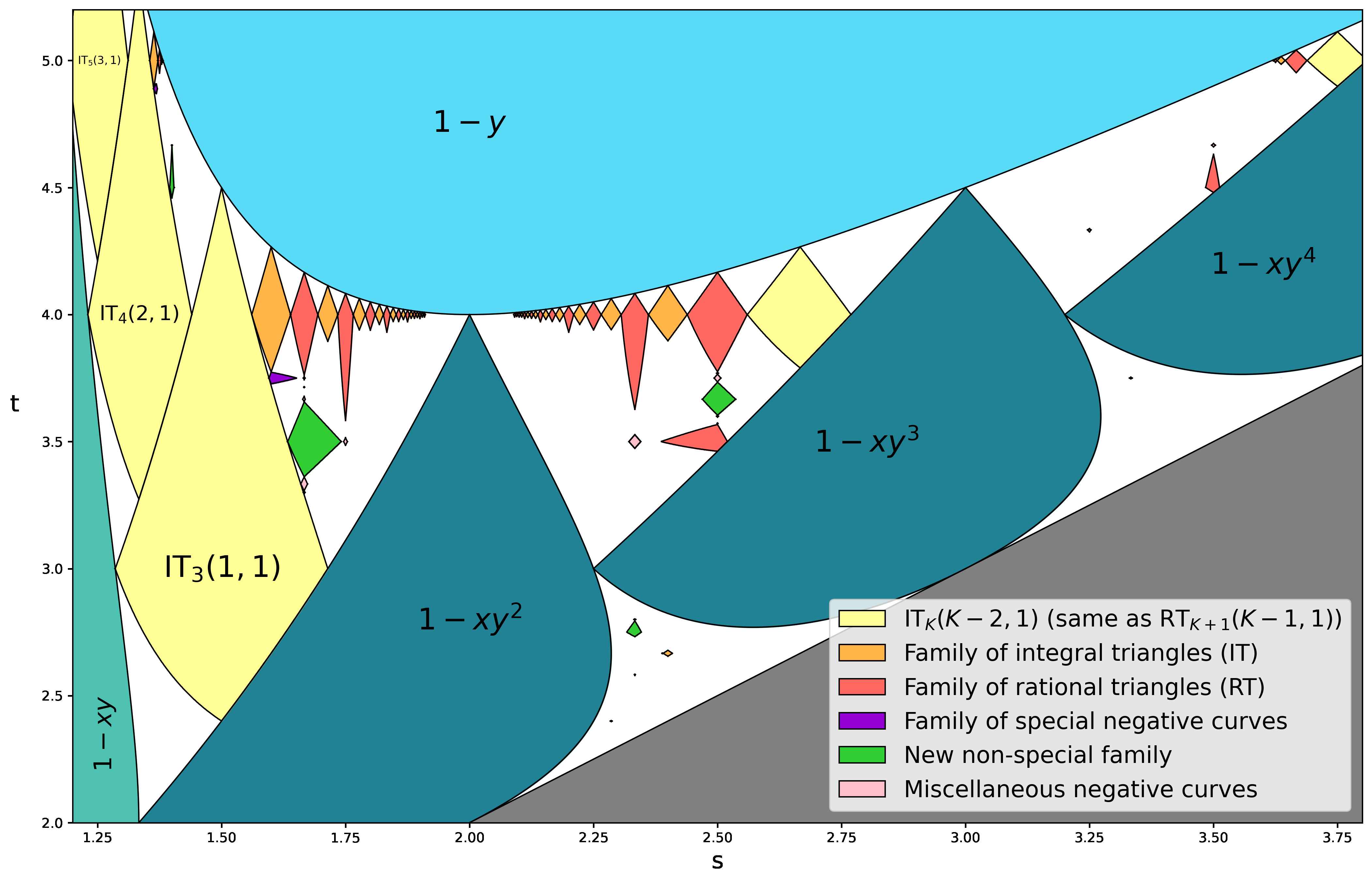}
  \caption{Negative curves in the parameter space of triangles.}
  \label{fig-plot}
\end{figure}

The rest of the article is about Figure~\ref{fig-plot}. It turns out that the location of the diamonds on the map explains many results and open problems in the field. We will briefly list some of these, with more details in subsequent sections.

\subsection{Infinite families of negative curves}

In our previous work \cite{GGK1, GGK2, GGK3}, we found infinite families of polynomials $\xi$ that define negative curves in some triangles. These families are unified in \cite{GGK3} as follows. For every integer $K\geq 4$, there are two infinite families of polynomials $\xi$ described by pairs of integers $(M_n, N_n)$, $n=1$,$2$,\ldots. The first family of polynomials is supported in the integral triangles $IT_K(M_n,N_n)$, the second family in rational triangles $RT_K(M_n,N_n)$. These negative curves can all be found on the map. The diamonds corresponding to $(K, M_n, N_n)$ have centres with coordinate $t=K$. Along each line $t=K$ the diamonds alternate: $IT_K(M_1,N_1)$, $RT_K(M_1, N_1)$, $IT_K(M_2,N_2)$, $RT_K(M_2, N_2)$, $\ldots$, $IT_K(N_2,M_2)$, $RT_K(N_2,M_2)$, $IT_K(N_1,M_1)$, $RT_K(N_1,M_1)$.

We will describe below a new infinite family of negative curves. For each $K\geq 4$ there is a triangle with slopes
\[ s=\frac{2K-3}{2K-5}, \quad t=K-\frac{1}{2} \]
that supports a negative curve with vanishing order 
\[ m= \left\lceil K^2-\frac{7}{2} K +2 \right\rceil.\]

We summarize in Table~\ref{tab-infinite} all known infinite families of negative curves. The family of special negative curves is described in the next subsection.

\bgroup
\def\arraystretch{1.5}%  1 is the default, change whatever you need
\begin{table}[ht]
  % used for centering table
  \begin{tabular}{| c | c | c | c |}
    % centered columns (4 columns)
    \hline
    Name &  Order of vanishing $m$ & $s$ & $t$ \\ %[0.3ex]
    % inserts table
    % heading
    \hline
    \hline %inserts double horizontal lines
    % inserts single horizontal line
    $IT_K(M,N)$ & $M+N$ & $\frac{NK}{M+N}$& $K$ \\
    $RT_K(M,N)$ & $M$ & $\frac{M+N}{M}$ & $K$ \\
    New non-special & $\lceil K^2-\frac{7}{2} K +2 \rceil$ & $\frac{2K-3}{2K-5}$ & $K-\frac{1}{2}$\\
    New special & $(K-3)(K-1)^2$ & $\frac{K(K-2)}{(K-1)(K-2)-1}$ & $K-\frac{1}{(K-2)^2}$\\
    \hline
  \end{tabular}
   \\[2ex]
  \caption{Infinite families of negative curves. }
  \label{tab-infinite}
\end{table}
\egroup

There are a few other negative curves that can be found using computer search, but we have not been able to fit them into infinite families. (See Table~\ref{tab-add} on page~\pageref{tab-add}
below for a list of these curves.) In general, for each vanishing order $m$ there is a finite number of isomorphism classes of negative curves. For $m=1,2$ there is a unique curve, for $m=3$ there are two curves, for $m=4,5,6$ there are $4$ curves. All these curves for $m\leq 6$ can be found on the map. Computer experiments suggest that as $m$ gets larger, it becomes more difficult to find a negative curve that is not in one of the infinite families. 
Among the negative curves that we have found which do not yet fit into an infinite family, the largest value of $m$ that occurs is $m=99$.

\subsection{Special negative curves.}

To find a negative curve $C$, we need an irreducible polynomial vanishing to order at least $m$ at $e$ and supported in a triangle $\Delta$ with area $\leq \frac{m^2}{2}$. The vanishing condition gives ${m+1\choose 2}$ linear equations in the coefficients of $\xi$. Thus, if $\Delta$ contains at least ${m+1\choose 2}+1$ lattice points, then a polynomial vanishing to order at least $m$ can always be found. This polynomial may be reducible. If it is irreducible we call the polynomial $\xi$ (or the negative curve $C$ it defines) {\em special} if the triangle $\Delta$ contains fewer than ${m+1\choose 2}+1$ lattice points. In that case we call the difference between ${m+1\choose 2}+1$ and the number of lattice points in $\Delta$ the {\em deficiency} of $\xi$. 

Two special negative curves were found by Kurano and Matsuoka \cite{KuranoMatsuoka}. These are the unique negative curves in the blowups of $\PP(8,15,43)$ and $\PP(5,33,49)$. They have vanishing orders $m=9$ and $m=18$, respectively, and deficiency $1$. These two weighted projective planes correspond to triangles with slopes $(s=1.6, t=3.75)$ and $(s=1.6\overline{6}, t=3.3)$ in our map.

We generalize the first example of Kurano and Matsuoka to an infinite family of special negative curves, one for each $K\geq 4$. The minimal triangles for these curves have slopes 
\[s=\frac{K(K-2)}{(K-1)(K-2)-1}, \quad t=K-\frac{1}{(K-2)^2},\]
vanish to order $m= (K-3)(K-1)^2$ and have deficiency $\frac{(K-2)(K-3)}{2}$.

We do not know if the $m=18$ example of Kurano and Matsuoka can be extended to an infinite family.

The specialty of a curve is related to an interpolation problem. A curve defined by $\xi$ is special iff there is a degree $\leq m-1$ bivariate polynomial that vanishes at all lattice points of the triangle $\Delta$. In general, the deficiency is the dimension of the space of polynomials of degree $\leq m-1$ that vanish at all lattice points of $\Delta$. In the infinite family of special curves mentioned above, the deficiency grows quadratically with $K$. Hence, the triangles of these curves have more and more relations among their lattice points.

Kurano \cite{Kurano21} studied the genus of negative curves, noting in particular that all known negative curves are rational. Similarly, all negative curves that we have listed are rational, except possibly the special curves. Consider a negative curve defined by $\xi$, and let $Y$ be the blowup of the toric variety defined by the Newton polygon of $\xi$. Here we assume $m>1$ to avoid the case where the Newton polygon is a line segment. Then, $\xi$ defines a negative curve $C$ that lies in the smooth locus of $Y$. One can compute that $K_Y \cdot C<0$ for all known negative curves except the special ones. This implies that such curves are rational and smooth in $Y$. For the special family we find that $K_Y \cdot C= (K-1)(K-4)\geq 0$, and the arithmetic genus of the curve in $Y$ is $\frac{(K-1)(K-4)}{2}$. We do not know the geometric genus of these curves when $K\geq 5$.

\subsection{No negative curves.}

It is possible that for some triangle $\Delta$ there is no negative curve at all. In that case the boundary of the Mori cone lies on a ray with self-intersection number zero. Kurano and Matsuoka \cite{KuranoMatsuoka} computed that the blowup of $\PP(9,10,13)$ does not have a negative curve for $m\leq 24$. One can also see, by computation or from the Pick's theorem, that as $m$ grows larger, a negative curve with vanishing order $m$ must have higher and higher deficiency. This strongly suggests that this blowup has no negative curve, but we do not have a proof of it. Another interesting fact is that if there is no negative curve in $X$, then the Mori cone of $X$ must have an edge with irrational slope.

We conjecture that most of the $st$-plane that is not already covered by diamonds contains no negative curve. Nevertheless, we do not have a proof of this for any triangle.

\subsection{Mapping $\PP(a,b,c)$.}

There is a large body of work in commutative algebra and algebraic geometry on the MDS property of blowups of $\PP(a,b,c)$. For example, Huneke \cite{Huneke}, Cutkosky \cite{Cutkosky} and Srinivasan \cite{Srinivasan} have shown that these blowups are always MDS when one of the weights equals $1,2,3,4$ or $6$. In particular, such blowups contain a negative curve. 

In a recent work, McKinnon, Razafy, Satriano and Sun \cite{MRSS} compute bounds for the nef cone of $\Bl_e\PP(a,b,c)$ for small values of $a$. 

It is therefore interesting to see where the weighted projective planes $\PP(a,b,c)$ lie on the map and if we can find negative curves in them.

We will show that if $a$ is fixed, then all $\PP(a,b,c)$ correspond to points on the $st$-plane with either $s$ or $t$ of the form $\frac{a}{d}$, where $1\leq d\leq a-1$ and $\gcd(a,d)=1$. There are further symmetries that can be used to reduce the number points on the $st$-plane that we need to consider. For example, the $a=4$ triangles lie on the horizontal line $t=4$ and on the vertical lines $s=\frac{4}{3},\frac{4}{1}$. The ones on the horizontal line $t=4$ contain a negative curve from one of the infinite families $IT_4(M_n,N_n)$, $RT_4(M_n, N_n)$. The ones lying on the vertical lines can only contain the negative curve with $m=1$ defined by a binomial. One can similarly find negative curves in all cases when $a<4$ or $a=6$, see Example~\ref{classical-results}. Moreover, once the negative curves have been identified, the MDS property for these spaces follows from Proposition~\ref{prop-CD}.

It is not known if all blowups of $\PP(a,b,c)$ with $a=5$ are MDS or even if they contain a negative curve. In our picture, all triangles corresponding to $a=5$ are equivalent to the ones on the horizontal line $t=5$ or on the vertical line $s=5/3$. The ones on the line $t=5$ contain a negative curve from the infinite families $IT_5(M_n,N_n)$ and $RT_5(M_n,N_n)$, and they are all MDS. The vertical line $s=5/3$ is the interesting one. 
See Figure~\ref{fig-m2} on page~\pageref{fig-m2} below for a closer look at this vertical line. The line contains the curve in $RT_4(3,2)$ and the special curve with $m=18$ of Kurano and Matsuoka. 
Along this line there is a sequence of diamonds with some gaps between them. We suspect that most triangles in these gaps do not contain negative curves.  Srinivasan \cite{Srinivasan} studied the case $a=5$ and gave bounds on $b/c$ for the blowup to be a MDS. These bounds can be related to the location of the triangle on the line $s=5/3$.

\subsection{The MDS property}

We fix one diamond corresponding to a negative curve $C^\circ \subset T$ and ask which triangles in it have the MDS property. One expects that the MDS region in each diamond is divided into pieces corresponding to different curves $D$.

It turns out that the MDS property in a diamond is often determined by neighbouring diamonds. Suppose a triangle $\Delta$ lies in the intersection of two diamonds. This $X$ contains two irreducible curves with zero self-intersection: the negatives curve of each diamond. Hence, $X$ is a MDS. The MDS property is sometimes preserved as we move from the intersection point into the interior of one of the diamonds. The reason is that the two curves that were disjoint in $X$ sometimes remain disjoint; one of them becomes the negative curve $C$, the other one the curve $D$. 

In our previous work \cite{GGK1, GGK2, GGK3} we have determined the MDS property in some cases for the negative curves in $IT_K(M,N)$ and $RT_K(M,N)$. It turns out that in all these cases the MDS property can be explained by the intersection of diamonds. For example, as we move along the horizontal line $t=K$, the diamonds for $IT_K(M,N)$ and $RT_K(M,N)$ are stringed together. As we move from one diamond to the next, the $C$ and $D$ curves switch their roles. For triangles in the diamond of $IT_K(M_n,N_n)$, $n>1$, we can determine completely which ones are MDS and which ones are not. For the remaining triangles in the two infinite families we can do this only for the upper half of the diamond.

The intersection of diamonds can be used to prove the MDS property in many other cases. This includes the minimal triangles supporting the special curves, and the new family of non-special curves for $K$ even. In addition to these examples, we will prove that all triangles on the horizontal line $t=K\in \ZZ$ have the MDS property. In the region where the negative curve is defined by $\xi=1-y$, it is interesting that points on the line $t=K$ with different $s$ value have a different $D$ curve (even when restricted to the torus). Indeed, the Newton polygons of the polynomials defining the curves $D$ have left edge with slope $s$. This shows somewhat unexpectedly that the interval is not divided into discrete subintervals corresponding to different $D$. Instead, every rational point corresponds to a different polynomial defining the curve $D$.  

We refer again to Figure~\ref{fig-m2} on page~\pageref{fig-m2} below for some of the intricacies involved in determining the MDS regions in each diamond.

\subsection{Positive characteristic.} We work over a field $k$ of characteristic $0$. However, one can draw a similar picture of diamonds when working over any field. In positive characteristic the existence of a negative curve implies that the blowup is a MDS \cite{Cutkosky}. Negative curves in characteristic $0$ can all be defined by integer coefficients. It follows from the proofs given below that the four infinite families of curves stay irreducible in any characteristic. Hence, these same diamonds appear when we work in positive characteristic. Some negative curves in characteristic $0$ are reducible in characteristic $p$. For example, the curve with $m=16$ in  Table~\ref{tab-add} on page~\pageref{tab-add} decomposes as a union of an irreducible negative curve with $m=15$ and a curve with $m=1$ in characteristic $2$. In positive characteristic there are in general more negative curves. For example, the blowup of $\PP(9,10,13)$, which is not known to contain a negative curve in characteristic $0$, does contain negative curves in characteristics $2$ and $3$.

\section{The Parameter space of triangles}

As explained in the introduction, we consider rational triangles $\Delta$ in $\RR^2$ with one edge horizontal and the other two edges with slopes $s$ and $t$. In this section we study the $st$-plane that parametrizes these triangles. We call two triangles isomorphic if they differ by a translation, scaling, and action of $\GL_2(\ZZ)$ on the plane $\RR^2$. Two triangles are isomorphic if and only if the corresponding toric varieties are isomorphic. 

For a more complete picture, one should consider the slopes of all three edges of $\Delta$. For each triangle we can choose coordinates so that one of its edges is horizontal. This implies that every triangle is isomorphic to one with slopes $\{0, s, t\}$, corresponding to the point $(s,t)$ or $(t,s)$ in our parameter space. However, we should warn that considering the $st$-plane only has several shortcomings. The most serious of them being that making an edge horizontal is not a continuous operation: two triangles may have slopes very close to each other, but the corresponding points on the $st$-plane can be very far apart. Nevertheless, we consider the $st$-plane as a good way to visualize the parameter space.

\subsection{The fundamental domain}

We consider triangles with one edge horizontal, and isomorphisms between these triangles that map the horizontal edge to another horizontal edge. Let us first find a fundamental domain on the $st$-plane containing one point for each isomorphism class.

We will make one more restriction by ignoring the triangles where $s$ and $t$ have different signs, or where one of them is infinite. Such triangles are very simple for the problem we consider. In these cases the variety $X=\Bl_e X_\Delta$ contains two non-intersecting curves, defined by the vanishing of $1-y$ and $1-x$. Hence, $X$ is always a MDS with the negative curve defined by either $1-y$ or $1-x$. 

\begin{lemma} A fundamental domain in the coordinates $u=1/s, v=1/t$ can be chosen as the triangle in the $uv$-plane with vertices
  \[(0,0), (1,0), (1/2,1/2).\]
\end{lemma}
\begin{proof}
  Since we consider $u$ and $v$ as an unordered pair, we may order them so that $u>v$. The operations of translation and scaling do not change the slopes. The subgroup of $GL_2(\ZZ)$ that preserves the zero slope is generated by the shear transformation
  \[ (u,v) \mapsto (u+1,v+1),\]
  and multiplication with $-1$, which after reordering so that $u>v$ can be written as
  \[ (u,v)\mapsto (-v,-u).\]
  This last operation is the reflection of the plane across the line $u=-v$. 

  A fundamental domain for this action is the strip 
  \[ 0 \leq u+v \leq 1. \]
  Restricting to the half-plane $u > v$ and removing points of the form $u> 0, v < 0$ gives the triangle as in the statement of the lemma.
\end{proof}

Figure~\ref{fig-plot-inv} shows the fundamental domain in the coordinates $u,v$. 
In the coordinates $s,t$ the fundamental domain is defined by 
\[ t \geq s \geq 1, \quad t\geq 1+\frac{1}{s-1}.\]
For a point in this domain, we may assume that the base of the triangle is horizontal, $s$ is the slope of the left edge and $t$ the slope of the right edge. 

\begin{figure}
  \centering
  \includegraphics[width=\textwidth]{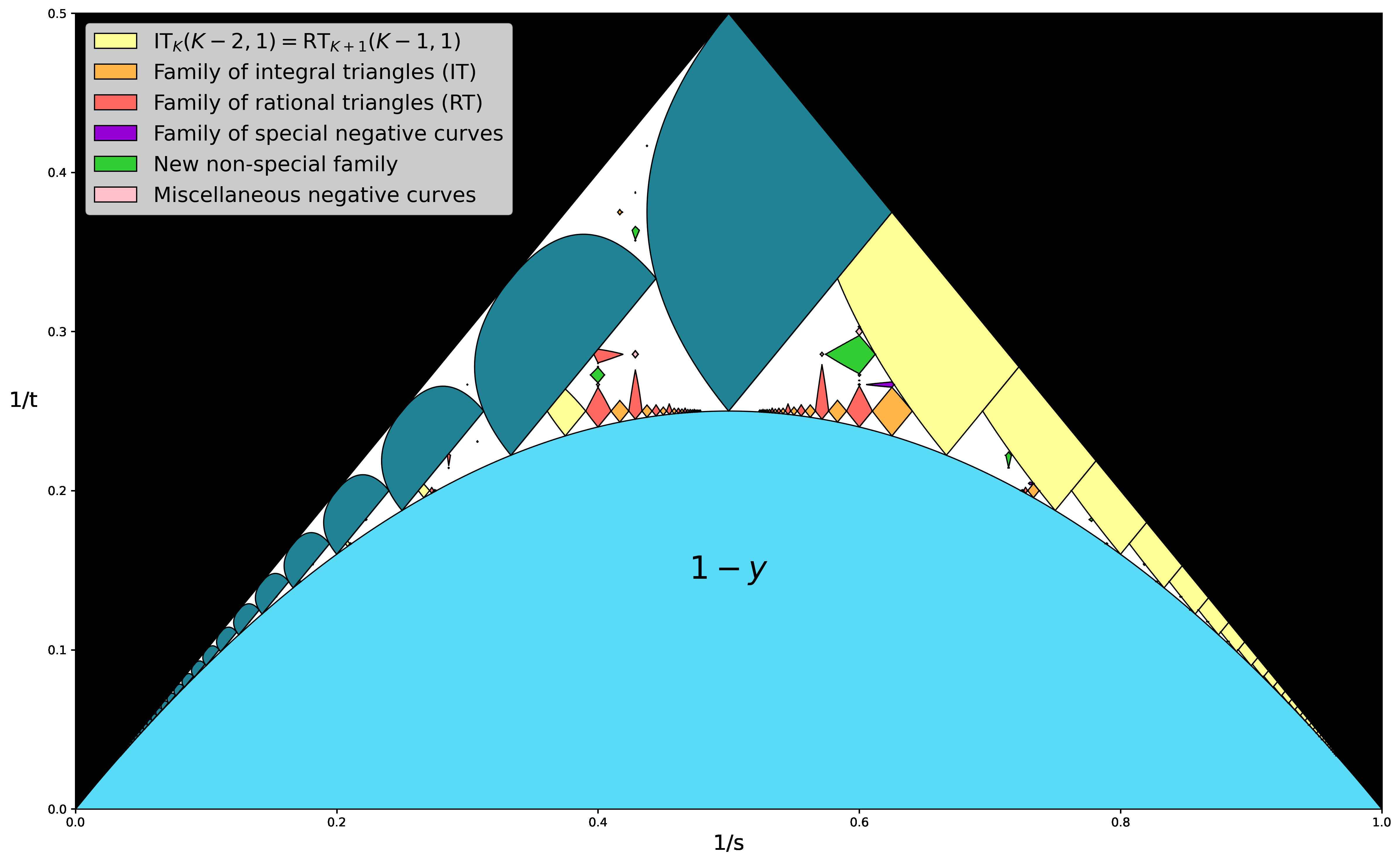}
  \caption{Negative curves in the parameter space of triangles.}
  \label{fig-plot-inv}
\end{figure}

% \subsection{Other isomorphisms}
Any triangle corresponds to a unique point $(s,t)$ in the fundamental domain above once we choose which edge we make horizontal. Since there are three choices for the horizontal edge, there are up to three points that correspond to the same triangle. We will consider some examples of this.

\begin{example}
  Consider the set of points $(s,t)$ where $t=K$ is an integer. We can make the edge with slope $t$ horizontal by applying a shear transformation in the vertical direction. The resulting triangle has coordinates $(t-s, t)$. Thus, the triangles with $K/2\leq s<K$ are isomorphic to triangles with $0<s\leq K/2$. Moreover, if $s \leq 1$ then we can transform the triangle to one that has $s$ and $t$ with different signs or one of them infinite. Hence, when $t=K$ it is enough to consider points with $1<s<K/2$. This isomorphism explains the symmetry in Figure~\ref{fig-plot} on the horizontal line $t=K$.

If $s=L$ is an integer, we can similarly make the edge with slope $s$ horizontal. The resulting triangle has slopes $\{s-t, s\}$, where $s>0$ and $s-t<0$, so we may ignore these triangles.

The triangle with slopes $s=2.25, t=3.5$ has the property that no matter which edge we make horizontal, we get the same point $(s,t)$. This triangle appears in the first entry of Table~\ref{tab-add} below as the minimal triangle supporting a negative curve.  
\end{example}

\section{Weighted projective planes $\PP(a,b,c)$.}

Weighted projective planes $\PP(a,b,c)$ are special cases of the toric varieties $X_\Delta$. We consider here which points $(s,t)$ correspond to weighted projective planes and how to determine $a,b,c$ from $s,t$. 

For a rational triangle $\Delta$, let $\vec{n}_i$, $i=1,2,3,$ be its outer normal vectors, which we choose to be integral and primitive. Recall that $X_\Delta$ is isomorphic to $\PP(a,b,c)$ if the normal vectors generate the lattice $\ZZ^2$, in which case $a,b,c$ are given by a relation
\[ a \vec{n}_1 + b \vec{n}_2 + c \vec{n}_3=0.\]
The positive integral vector $(a,b,c)$ is usually also chosen primitive.

\begin{lemma} 
  Let the triangle $\Delta$ have rational slopes $0 < s <t$. Write $s = a/p$, $t=b/q$ in lowest terms. Then $X_\Delta$ is isomorphic to a weighted projective plane if and only if $a$ and $b$ are relatively prime. In this case $X_\Delta \isom \PP(a,b,c)$, where $c=bp-aq$. 
\end{lemma}
\begin{proof}
  The normal vectors are $(b,-q)$, $(-a,p)$, $(0,-1)$. These generate the lattice $\ZZ^2$ if and only if $a$ and $b$ generate the lattice $\ZZ$. The three vectors satisfy the relation with coefficients $a,b,c=bp-aq $.
\end{proof}

The previous lemma tells us exactly which points $s,t$ correspond to weighted projective planes. Conversely, given a weighted projective plane $\PP(a,b,c)$, it corresponds to a triangle $\Delta$. We can make the side corresponding to $c$ horizontal. Then the other sides must have slopes $a/p$ and $b/q$ for some $p$ and $q$. Consider now a fixed $a$ and varying $b$ and $c$. All such $\PP(a,b,c)$ must have either $s$ or $t$ of the form $a/q$ for some $q$ that is relatively prime to $a$. If the point $(s,t)$ lies in the fundamental domain, then $1\leq q <a$. 

It may happen that the point $(s,t)$ corresponding to some $\PP(a,b,c)$ has $s$ and $t$ with different signs and hence cannot be moved to the fundamental domain. This case is well-known and often discarded in the commutative algebra literature. It means that the point $e=(1,1)$ that is blown up is a complete intersection in $\PP(a,b,c)$, with homogeneous ideal generated by two elements, $1-x$ and $1-y$ in our notation.

\begin{example}\label{classical-results}
  Consider the weighted projective planes $\PP(a,b,c)$ where $a=2,3,4,6$. These correspond to points where either $s$ or $t$ is equal to $a$ or $a/(a-1)$. The cases where $s$ is an integer are isomorphic to triangles where $s$ and $t$ have different signs and are hence not interesting. When $t$ is an integer, we may further restrict $s\leq t/2$. Thus, all interesting examples of such $\PP(a,b,c)$ have coordinates $t=a, 1<s\leq t/2$ or $s=a/(a-1), t>a$. In Figure~\ref{fig-plot} these points lie on the horizontal line $t=a$ covered by the diamonds of $IT_a(M,N), RT_a(M,N)$ and on the vertical line $s=a/(a-1)$ covered by the diamond of $IT_a(a-2,1)$ and the diamond of $\xi=1-y$.

\end{example}

\begin{example}
  Consider the weighted projective planes $\P(a,b,c)$ with $a=5$. These correspond to points $(s,t)$ where either $s$ or $t$ has the form $5/q$, where $q=1,2,3,4$. We proceed to explain how these cases are isomorphic to ones with $s=5/3, t>2$ and $t=5, 1<s<5/2$.
  
%  We now consider cases where $s$ or $t$ has the form $5/q$ for some integer $q=1,2,3,4$. These examples correspond to weighted projective planes $\PP(5,b,c)$.
%  
  First consider points $(s,t)$ where $s=5/2$. Choosing coordinates appropriately, we can make the edge with slope $s$ horizontal. The other two edges will then have slopes $5/3$ and $2+1/(t-3)$. This shows that triangles on the vertical line $s=5/2$ are isomorphic to triangles on the vertical line $s=5/3$. The interesting line segments here are $s=5/2, t> 3$ and $s=5/3, t>2$. These are mapped to each other bijectively. The line segments $s=5/2, 5/2<t<3$ and $s=5/3, 5/3<t<2$ are mapped to points where $s$ and $t$ have different signs and hence can be ignored.
  
Now consider the case $t=5/2$. This lies in the fundamental domain if $5/3<s<5/2$. By the group action these are isomorphic to $s=5/3$, $5/3<t<5/2$, which we considered before.

When $s=5$, we can make this side horizontal and get the two slopes with different signs. These can be ignored.  

When $s=5/4$ and we make this edge horizontal, the remaining two slopes are $t=5, s=4-1/(t-4)$. 

  In summary, all interesting cases of $(s,t)$ in the fundamental domain where either $s$ or $t$ is equal to $5/q$ for some $q=1,2,3,4$ are isomorphic to triangles with coordinates $s=5/3, t>2$ and $t=5, 1<s<5/2$. In Figure~\ref{fig-plot} the horizontal line $t=5$ is covered by diamonds of $IT_5(M,N), RT_5(M,N)$. The vertical line $s=5/3$ is magnified in Figure~\ref{fig-m2} below. The vertical line contains the vertical axis of $RT_4(3,2)$ and several other diamonds, with gaps between them.

  It is interesting that the values of $b/c$ not covered by Srinivasan's bounds in \cite{Srinivasan} exactly lie on the vertical line $s=5/3$. Moreover, the slope $t$ of these triangles is a function of $b/c$ and the bounds given by her precisely correspond to the bottom vertex of the diamond of $RT_{4}(3,2)$ and the top boundary point of the diamond of $IT_{3}(1,1)$.

\end{example}

\section{Infinite families of negative curves}

In \cite{GGK3} we described two infinite families of negative curves, each family depending on two parameters. We start by recalling these families. We then construct a new infinite family of non-special negative curves. 

\subsection{The families $IT_K(M,N)$ and $RT_K(M,N)$.}\label{subsection-IT-RT}

Consider positive integers $K$, $M$ and $N$ satisfying the equation 
\begin{equation} \label{eqn-MN}
  (M+N)^2 = KMN+1.
\end{equation}
For $K=3$ there is a unique pair $(M,N)=(1,1)$, and for $K\geq 4$ there is an infinite number of pairs $(M,N)$ satisfying the equation. We describe these pairs below. Since $K$ is determined by $M$ and $N$ if $M,N>0$, we will often omit it from the notation.

For each $K, M, N$ satisfying the Equation~\ref{eqn-MN} there are two negative curves. The first one is supported in the integral triangle $IT_K(M,N)$ with vertices
\[ (0,0), (M,0), (M+N, KN).\]
This triangle has coordinates $s=KN/(M+N), t=K$, and the negative curve has vanishing order $m=M+N$. The second curve is supported in the rational triangle $RT_K(M,N)$ with vertices
\[ (0,0), \left(M-\frac{M+N}{K},0\right), (M,M+N).\]
The triangle has coordinates $s=(M+N)/M, t=K$, and the negative curve has vanishing order $m=M$. 

On the $st$-plane the triangles $IT_K(M,N)$ and $RT_K(M,N)$ lie on the horizontal line with $t=K$. The curves with $M>N$ are isomorphic to those with $M<N$. Nevertheless, both diamonds appear on the $st$-plane because points in one diamond may not be isomorphic to points in the other. 

For small values of $M$ and $N$ there are isomorphisms between pairs of these curves. For example, the curve in $IT_{K-1}(K-3,1)$ is isomorphic to the curve in $RT_{K}(K-2,1)$. Similarly, all curves with $m=1$ are isomorphic. This includes $IT_K(0,1)$ and $RT_K(1,K)$ for all $K$. Other than these coincidences, the negative curves are pairwise non-isomorphic. 

We will recall how to find all solutions $K,M,N$ to Equation~\ref{eqn-MN}, and how to construct the equations of the negative curves recursively. These equations will be used below.

Fix $K\geq 3$ and consider the sequence of integers $F_0=0, F_1=1, F_2=K-2, \ldots$ defined by the recurrence
\[ F_{n+2} = (K-2) F_{n+1} - F_n.\]
Then the non-negative integral solutions $M>N$ of the equation above are
\[ (M_n, N_n) = (F_{n+1}, F_n), \quad n=0,1,\ldots.\]
Since the equation is symmetric, $(M=N_n, N=M_n)$ is also a solution. However, we have isomorphisms of triangles and hence isomorphisms of negative curves
\[ IT(N_n,M_n) \isom IT(M_n, N_n), \quad RT(N_n, M_n) \isom RT(M_{n-1}, N_{n-1}).\]

We let $\xi_n^{int}$ and $\xi_n^{rat}$ be the polynomials defining the negative curves for $IT(M_n,N_n)$ and $RT(M_n,N_n)$, respectively, normalized so that their constant terms are $1$. We will sometimes use the notation $\xi_{n,K}^{int}$ or $\xi_{n,K}^{rat}$ to make the value of $K$ explicit.

For every $n>0$ these polynomials satisfy the recurrence relations
\begin{align}        \label{eqn.recurrence.relations} 
\begin{split}
  \xi^{int}_{n} &= \xi^{rat}_{n} \xi^{rat}_{n-1} -\varepsilon^{rat}_{n} x^{M_n}(y-1)^{M_n+N_n},\\
  \big(\xi^{rat}_{n-1}\big)^K  &=  \xi^{int}_{n} \xi^{int}_{n-1} -\varepsilon^{int}_{n} x^{M_n+N_n}(y-1)^{KN_n}.
\end{split}
\end{align}
Here $\varepsilon^{rat}_{n}$ and $\varepsilon^{int}_{n}$ are the coefficients of the highest degree terms of the polynomials. For $K$ even they are
\[ \varepsilon^{int}_{n} = -1, \quad \varepsilon^{rat}_{n} = (-1)^{n+1}.\]
For $K$ odd,
\[ \varepsilon^{int}_{n} =  \begin{cases}\hphantom{-}1 & \text{if $n\equiv 1 \mod 3$,}\\ -1 & \text{otherwise,} \end{cases} \qquad
  \varepsilon^{rat}_{n} =  \begin{cases}\hphantom{-}1 & \text{if $n\equiv 2 \mod 3$,}\\ -1 & \text{otherwise.} \end{cases} \]
From these equations and the initial conditions $\xi^{int}_{0} = 1-x$ and $\xi^{rat}_{0} = 1-xy$ one can compute all these polynomials recursively.

\subsection{A new family of negative curves.} \label{sec-new-nonspecial}

We describe here a new infinite family of negative curves, complementing the two-parameter families described above. 

Let us denote the parameters
\[ a=2K-3,\quad b=2K-1, \quad c=4K^2-16K+11, \quad d= 2K-5, \quad m=\left\lceil K^2-\frac{7}{2}K+2\right\rceil.\]

\begin{theorem} 
  For each $K \geq 4$ there is a negative curve in the blowup of $\PP(a,b,c)$, with $m$ the order of vanishing at $e$. The triangle $\Delta_K$ corresponding to the weighted projective plane has horizontal base, and left and right slopes
  \[ s= \frac{a}{d} = 1+ \frac{2}{2K-5},\quad t=\frac{b}{2} = K - \frac{1}{2}.\]

  For even $K$, the triangle has vertices 
  \[ \left(-\frac{1}{a},0\right), \quad (m-K+2,0),\quad \left(m, \frac{(K-2) b}{2}\right).\]
  The triangle has two integral vertices and the lower left vertex rational.

  For odd $K$, the triangle has vertices 
  \[ (0,0),\quad (m-K+2,0),\quad \left(\frac{K-1}{2}+\frac{1}{c}\right) (d,a).\]
  This triangle has integral base and rational top vertex.
\end{theorem}

The rest of this subsection consists of the proof of the theorem.

In both even and odd cases it is easy to check that the triangle contains exactly ${{m+1}\choose{2}} +1$ lattice points. To see this, one considers the convex hull of lattice points in the triangle and computes its area $(m^2-1)/2$ and lattice perimeter $m+1$. Then Pick's theorem gives the number of lattice points. From the number of lattice points and the area of the triangle being $<m^2/2$, we know that the triangle supports a negative curve. The only problem is to show that this curve is irreducible. We will do this below for even and odd cases separately. Note also that these negative curves are not isomorphic to the negative curves in the previous infinite families. The triangles $RT(M,N)$ have width exactly $m$. For these triangles it is not possible to choose coordinates in which their width is $m$.

A special feature of this family of triangles, as well as the triangles $IT(M,N)$ and $RT(M,N)$ above, is that the convex hull of the lattice points in the triangle has width exactly $m$. In fact, to get a nontrivial negative curve, the width has to be at least $m$.

\begin{lemma}\label{lemma-width}
  Let $f(x,y)\in k[x,y]$ be a polynomial that vanishes to order $m$ at $e=(1,1)$. Assume that the polynomial $f$ has degree less than $m$ in $x$. Then $1-y$ divides $f(x,y)$.
\end{lemma}

\begin{proof}
  The polynomial $f(x,1)$ must have multiplicity at least $m$ at $x=1$. Since the degree of this polynomial is less than $m$, it must vanish identically.
\end{proof}

Let us denote by $f(x,y)$ a nonzero polynomial supported in the triangle $\Delta= \Delta_K$ that vanishes to order at least $m$ at $e$. We will prove that this $f$ is irreducible. Let us start with the case of even $K$. 

\begin{lemma} 
  Assume that $K$ is even. Then $f(x,y)$ is irreducible.
\end{lemma}

\begin{proof}
  We will use the following simple fact. Suppose we have a possibly reducible curve $C$ in $X$ with negative self-intersection, and an irreducible curve $D$ such that $C\cdot D=0$. Then $C$ lies on the boundary of the Mori cone of curves and hence is an integer multiple of an irreducible negative curve. Indeed, the only effective curve classes in $X$ that multiply to zero must lie on the boundaries of the Mori cone and the nef cone.

  To find this curve $D$, we use the observation that on the $st$-plane the diamond of $f(x,y)$ meets the diamond of the negative curve defined by $\xi_1^{int}$ supported in $IT_{K-1}(K-3,1)$. Let us check that the curve $D\subset X$ defined by $\xi_1^{int}$ is orthogonal to $C$. We denote by $b_C, h_C$ the base length and the height of the triangle $\Delta$. Similarly, let $b_D, h_D$ be the base and height of the smallest triangle parallel to $\Delta$ that supports $\xi_1^{int}$. Then the classes of $C$ and $D$ are
  \begin{alignat*}{2}
    [C] &= H-mE, \\
    [D] &= \frac{b_D}{b_{C}} H - (K-2) E.
  \end{alignat*}
  Here $H$ is the class corresponding to the triangle $\Delta$, and $K-2$ is the order to which $\xi_1$ vanishes at $e$. Now using that $H\cdot H = b_C h_C$, we get
  \[ C\cdot D = b_D h_C - (K-2)m = b_D \frac{(K-2)(2K-1)}{2} - (K-2)m.\]
  We need to find $b_D$ and check that this intersection number is zero. The Newton polygon of $\xi_1^{int}$ is the triangle with vertices $(0,0),(K-3,0),(K-2,K-1)$, with the slope of the right edge $K-1$. To make the sides parallel to the sides of $\Delta$, we increase the left slope (which does not affect the base), and also increase the right slope to $K-\frac{1}{2}$. The new triangle will have base
  \[ b_D = (K-2) -\frac{2(K-1)}{2K-1}.\]
  Now one easily checks that $C\cdot D=0$.

  This leaves us with the possibility that $C$ is not irreducible, but it is a multiple of an irreducible curve. Let us rule this out. First notice that $C$ and $D$ must be disjoint. Indeed, if not, then $D$ would contain as a component the irreducible negative curve. But $D$ is irreducible and has positive self-intersection. Now since $D$ passes through the torus fixed points corresponding to the right and top vertices of the triangle $\Delta$, it follows that $C$ must not pass through these points, hence these lattice points are in the support of $f(x,y)$. From the fact that $1-y$ does not divide $f$ we get that the extreme left lattice point $(0,0)$ must also be in the support of $f$. In summary, the Newton polygon of $f$ has horizontal base of length $m-K+2$ and right edge of lattice length $\frac{K-2}{2}$. If $f=g^p$ for some integer $p>1$ then $p$ must divide the lengths of these sides. However, one easily checks that the lengths are relatively prime.
\end{proof}

Let us now turn to the case of odd $K$. We would like to prove as before that the Newton polygon of $f$ contains the two integral vertices and the highest lattice point $\frac{K-1}{2} (d,a)$. We can only prove this for the top lattice point and the left vertex by showing that $1-y$ does not divide $f$. This will be enough to prove that $f$ is irreducible.

\begin{lemma}
  Let $K$ be odd. Then $1-y$ does not divide $f$.
\end{lemma}

\begin{proof}
  We start by showing that if $1-y$ divides $f$ then also $(1-y)^d = (1-y)^{2K-5}$ divides $f$. If $1-y$ divides $f$ then the lattice points in the Newton polygon of $f$ lie in fewer than $m$ columns. Writing $f=(1-y)g$, the lattice points in the support of $g$ are obtained by starting with the lattice points in the support of $f$ and removing one point from each column. If the number of columns decreases, we can repeat this step.

  The triangle $\Delta$ has left slope 
  \[ \frac{2K-3}{2K-5} = 1+\frac{1}{K-\frac{5}{2}}.\]
  Counting the lattice points in $\Delta$ in the columns $x=0,1,2,\ldots$, we get the numbers:
  \[ 1, 2, 3, \ldots, K-2, K, K+1, \ldots, 2K-4, 2K-2,\ldots.\]
  Counting similarly columns from the right we get
  \[ 1, K-1,\ldots.\]
  Assuming that one of the extreme left or right columns with $1$ lattice point is not in the support of $f$, when we order the remaining columns in $\Delta$ by the number of lattice points, we get $1,2,\ldots, 2K-4,\ldots$. Every time we remove one point from each column, we decrease the number of columns. We can do at least $2K-4$ such steps, which is more than $2K-5$ that we need. (One can in fact prove that the columns of lattice points in $\Delta$ have size exactly $1,1,2,3,\ldots, m$. This can be used to prove the claim of the lemma. We will instead use a different argument that only needs the $2K-5$ shortest columns.)

  Let $C_1$ be the irreducible curve in $X$ defined by $1-y$. Note that $C$ is irreducible with positive self-intersection, which implies that it lies in the nef cone. Now write $C=(2K-5)C_1 + C_2$ for an effective curve $C_2$. Then, since $C_1$ lies in the nef cone, $C_1\cdot C_2\geq 0$. We will check that this inequality is not satisfied, giving a contradiction to the assumption that $1-y$ divides $f$. This is a straight-forward computation.

  Let $H$ be the curve class in $X_\Delta$ corresponding to the triangle parallel to $\Delta$ and the left side integral of lattice length $1$. The triangle has base of length 
  \[ B_H= \frac{c}{b}\]
  and vertices
  \[ (0,0), \quad (B_H,0),\quad (d,a).\]
  We can compute the self-intersection number
  \[ H^2 = B_H a = \frac{ac}{b}.\]
  The class of $C_1$ can be written in terms of the class of $H$ as:
  \[ \frac{bd}{ac} H - E.\]
  The class of $C_2$ is 
  \[ [C]-d[C_1] = \left(\frac{K-1}{2}+\frac{1}{c} - \frac{bd^2}{ac}\right) H - (m-d)E.\]
  Now the inequality $C_1 C_2\geq 0$ is equivalent to 
  \[ \frac{K-1}{2}+\frac{1}{c} - \frac{bd^2}{ac} \geq \frac{m-d}{d} = \frac{K-3}{2}\]
  This can be simplified to the inequality
  \[ bd^2-a \leq ac.\]
  Writing it out in terms of $K$ we get
  \[ (2K-1)(2K-5)^2-(2K-3) \leq (2K-3)(4K^2-16K+1),\]
  Which simplifies to 
  \[ -22K+25\geq 0.\]
  Clearly the inequality is false when $K\geq 5$.
\end{proof}

\begin{lemma}
  Let $K$ be odd. Then $f$ is irreducible.
\end{lemma}

\begin{proof}
  Suppose that $f$ is reducible, $f=f_1 f_2$, corresponding to $C=C_1+C_2$. 
  The Newton polygon of $f$ contains the left vertex $(0,0)$ of $\Delta$ and its highest lattice point $\frac{K-1}{2}(d,a)$, which also lies on the left edge. This integral edge of the Newton polygon of $f$ must be the sum of the corresponding integral edges of the Newton polygons of $f_1$ and $f_2$. It follows that the Newton polygons of $f_1$ and $f_2$ must have lower left vertex $(0,0)$, the highest lattice point on the left edge $w_i (d,a)$ for integers $w_1, w_2$ such that $w_1+w_2=\frac{K-1}{2}$. 

  Let $H$ be the curve class in $X_\Delta$ as in the previous proof, corresponding to a triangle with an integral left edge of lattice length $1$. Then the class of $C$ is 
  \[ [C] = \left(\frac{K-1}{2}+\frac{1}{c}\right)H - m E\]
  Since $1-y$ does not divide $f$, the Newton polygon of $f$ has width exactly $m=d\frac{K-1}{2}$. Then it must also be true that the width of the Newton polygon of $f_i$ is equal to its vanishing order. Hence, we may write the classes
  \[ [C_1] = \left(w_1+\frac{\varepsilon_1}{c}\right) H - dw_1 E,\]
  \[ [C_2] = \left(w_2+\frac{\varepsilon_2}{c}\right) H - dw_2 E,\]
  where $w_1, w_2 \geq 1$ are integers, $w_1+w_2=\frac{K-1}{2}$, and $\varepsilon_1 ,\varepsilon_2 \geq 0$ are rational numbers such that $\varepsilon_1 + \varepsilon_2 = 1$. 

  Another restriction we need to have is that the triangle corresponding to the class $\left(w_i+\frac{\varepsilon_i}{c}\right) H$ must be the smallest triangle parallel to $\Delta$ that contains the Newton polygon of $f_i$. This implies that the triangle must contain at least one lattice point on the right edge, and that means that $\varepsilon_i$ can only be $0$ or $1$. Let us assume that $\varepsilon_1 = 0$ and $\varepsilon_2=1$. Then $C_1$ has negative self-intersection and $C_2$ positive self-intersection.

  In the proof for the even $K$ we used the fact that the irreducible curve $D$ defined by $\xi_1^{int}$ has zero intersection with the negative curve. This is no longer true in the case of odd $K$. However, $D$ is an irreducible curve in $X$, and since it has positive self-intersection, it must lie in the nef cone. We get a contradiction to reducibility of $f$ by showing that $D\cdot C_1<0$.

  We compute the class of $D$ as in the case of even $K$. Let $B_D$, $B_H$ be the lengths of the bases of the triangles corresponding to the classes $D, H$. Then 
  \[ [D] = \frac{B_D}{B_H} H -(K-2)E.\]
  Here as before
  \[ B_D = (K-2) -\frac{2(K-1)}{2K-1}. \]
  Now using that $H^2= B_H a$, we compute
  \[ C_1\cdot D = \frac{B_D}{B_H} w_1 H^2 - dw_1(K-2) = B_D w_1 a - dw_1(K-2).\]
  The claim that this product is negative is equivalent to 
  \[ B_D a < d (K-2),\]
  which one checks easily.
\end{proof}

\subsection{Additional negative curves}

We list here negative curves that do not fit into the infinite families described above. These examples were found with the help of a computer. We only list non-special negative curves. The section below treats the special negative curves.

The finite set of negative curves is listed in Table~\ref{tab-add}. The columns of the table contain the vanishing order $m$, the slopes $(s,t)$ of the triangle, and the $\PP(a,b,c)$ whose blowup contains the negative curve. For some triangles it is more natural to give the triangle rather than the $\PP(a,b,c)$. For example, the negative curve with vanishing order $m=8$ is supported in a triangle that is three times the integral triangle with vertices $(0,0),(1,0),(3,7)$. This curve also appears as a negative curve in the blowup of $\PP(7,65,66)$, defined by a homogeneous polynomial of degree $1386$. The curve with $m=5$ is supported in a triangle with slopes $s=13/4, t=13/3$. This triangle does not correspond to a weighted projective plane. The toric variety is the quotient of $\PP^2$ by the cyclic group of order $13$. Deforming the right side slightly will give us a weighted projective plane $\PP(13, 186, 185)$. The negative curve is defined by a polynomial of degree $3344$ in the homogeneous coordinate ring.

\begin{table}[t]
  % used for centering table
  \begin{tabular}{| c | c | c | c | c |}
    % centered columns (4 columns)
    \hline
    $m$ & $(s,t)$ & Triangle & WPP & degree\\ %[0.3ex]
    % inserts table
    % heading
    \hline
    \hline %inserts double horizontal lines
    % inserts single horizontal line
    $4$ & $(2.25, 3.5)$ & & $\PP(7,9,10)$ & $100$ \\
    $5$ & $(3.25, 4.3\overline{3})$ & $\left[\left(-\frac{4}{13},0\right),\left(\frac{14}{13},0\right),\left(\frac{68}{13},18\right)\right]$ & $\PP(13, 186, 185)$ & $3344$ \\
    $6$ & $(1.8,3.5)$ & & $\PP(7,9,17)$ & $196$ \\
    $6$ & $(2.2, 3.5)$ & & $\PP(7,11,13)$ & $189$ \\
    $7$ & $(1.6\overline{6}, 3.6\overline{6})$ & & $\PP(5,11,18)$ & $220$\\
    $8$ & $(2.3\overline{3}, 3.5)$ & $3*[(0,0),(1,0),(3,7)]$ & $\PP(7,65,66)$ & $1386$\\
    $9$ & $(1.429, 4.5)$ & & $\PP(9, 10, 43)$ & $559$\\
    $11$ & $(1.6\overline{6}, 3.3\overline{3})$ & $2*[(0,0),(3,0),(6,10)]$ & $\PP(5, 47, 71)$ & $1420$\\
    $12$ & $(1.75, 3.3\overline{3})$ & & $\PP(7,10,19)$ & $437$ \\
    $15$ & $(1.75, 3.5)$ & $4*[(0,0),(2,0),(4,7)]$ & $\PP(7,114,227)$ & $6384$ \\
    $16$ & $(1.6\overline{6}, 3.714)$ & & $\PP(5, 26, 43)$ & $1196$ \\
    $21$ & $(1.3, 5.5)$ & & $\PP(11,13,84)$ & $2301$ \\
    $22$ & $(1.286, 5.6\overline{6})$ & & $\PP(9,17,92)$ & $2610$ \\
    $26$ & $(1.6\overline{6}, 3.75)$ & $3*[(0,0),(5,0),(9,15)]$ & $\PP(5, 139, 232)$ & $10440$ \\
    $31$ & $(1.4, 4.75)$ & & $\PP(7, 19,67)$ & $2926$ \\
    $99$ & $(1.4, 4.6\overline{6})$ & $10*[(0,0),(7,0),(10,14)]$ & & \\
    \hline
  \end{tabular}
  \\[2ex]
  \caption{Additional negative curves. }
  \label{tab-add}
\end{table}

One can list all negative curves up to isomorphism for small values of $m$. This can be done by classifying all integral polygons containing a small number of lattice points (see, for example, \cite{Castryck12}), and checking which one of them supports a negative curve. We list here all negative curves up to $m=6$.

For $m=1$ there is a unique negative curve defined by the polynomial $1-y$. For $m=2$ there is also a unique curve with coordinates $s=1.5, t=3$ on the map. This is the curve that is supported in $IT_3(1,1)$, $RT_4(2,1)$, $RT_4(2,3)$.

For $m=3$ there are two negative curves. The first one has coordinates $(1.3\overline{3},4)$ and is supported in $IT_4(2,1)$, $RT_5(3,1)$. The other one has coordinates $(1.6\overline{6}, 4)$ and is supported in $RT_4(3,2)$.  

For $m=4$ there are four non-isomorphic curves. These are $(1.25, 5)$ in $IT_5(3,1)$, $(1.75,4)$ in $RT_4(4,3)$, the curve in the table above, and the first curve for $K=4$ in the new family described in the previous section. 

For $m=5$ there are four curves, $(1.2,6)$ in $IT_6(4,1)$, $(1.6,4)$ in $IT_4(3,2)$, $(1.8, 4)$ in $RT_4(5,4)$, and the curve in the table. 

For $m=6$ there are again four curves, in $IT_7(5,1)$, $RT_4(6,5)$, and the two curves in the table.

For $m=7$ we have found four curves, in $IT_8(6,1)$, $RT_4(7,6)$, $IT_4(4,3)$, and the one in the table. There may be more curves with $m=7$.

\section{Special negative curves}
Consider a triangle $\Delta$ supporting a negative curve with equation $\xi$ vanishing to order $m$ at $e$. Recall that we call $\xi$ (or the negative curve $C$ it defines) special if $\Delta$ contains less than ${m+1\choose2}+1$ lattice points. Its deficiency is the difference between ${m+1\choose 2}+1$ and the number of lattice points in $\Delta$.

Kurano and Matsuoka first discovered two examples of special negative curves with deficiency $1$ by numerical methods. These live in the blowups of $\PP(8,15,43)$ and $\PP(5,33,49)$, vanish to orders $m=9$ and $m=18$ and lie in degrees $645$ and $1617$, respectively.

In this section we present the construction of an infinite family of special negative curves. This family generalizes the $m=9$ curve of Kurano-Matsuoka and has the notable property that the deficiency of its members becomes arbitrarily large.

\begin{theorem}\label{thm-special}
  Let $K\geq 4$ be an integer and define $N=K-2$ and $M=N^{2}-1$. Consider the weights
  \[
    (a,b,c) = \left(KN,KN^{2}-1, KMN^{2}-(M+N)\right).
  \]
  Then, $X=\Bl_{e}\PP(a,b,c)$ is a MDS.

  The negative curve $C$ in $X$ is supported in the triangle $\Delta$ with top vertex $(N(M+N)-1,KN^{2}-1)$, bottom RHS vertex $(MN-1,0)$ and LHS slope $\frac{KN}{M+N}$. It's bottom LHS vertex is $\left( -\frac{K-1}{KN},0\right)$.
  % \[
  %   \left(-\frac{K-1}{KN},0\right),(N(M+N)-1,KN^{2}-1)\quad\text{and}\quad (MN-1,0).
  % \]  
  Furthermore, this curve is special, vanishes to order $m=N(M+N)-1$ at $e$ and has deficiency $\frac{(K-2)(K-3)}{2}$.
\end{theorem}

For $K=4$ one recovers Kurano and Matsuoka's $(a,b,c)=(8,15,43)$ example. The first new curve in the family appears as soon as $K>4$. For $K=5$ we obtain a negative curve with $m=32$ and deficiency $3$, living in $\Bl_{e}\PP(15, 44, 349)$. For $K=6,7$ the weights are $(a,b,c)=(24, 95, 1421),(35, 174, 4171)$ with negative curves having $m=75,144$ and deficiencies $6$ and $10$, respectively.

\begin{remark}
  Unlike the construction of the infinite family presented above, this one is constructed by explicitly describing its defining equations. In the notation of Subsection~\ref{subsection-IT-RT}, the values of $M$ and $N$ in the theorem correspond to $M=M_{2}$ and $N=N_{2}$ with RHS slope $K$.

  Below we will use the polynomials $\xi_{2,K}^{int}$ and $\xi^{int}_{1,K-1}$ for the proof; here the first subindex corresponds to $n$ as in Subsection~\ref{subsection-IT-RT}, while the second one is the RHS slope of the corresponding triangle. More concretely, the first one of these polynomials is supported on the triangle $IT_{K}(M_{2},N_{2})=IT_{K}((K-2)^{2}-1,K-2)$, having RHS slope $K$. On the other hand, $\xi_{1,K-1}^{int}$ is supported on the triangle $IT_{K-1}(M_1,N_{1}) = IT_{K-1}(N_{2}-1,1)$ and has RHS slope $K-1$.
\end{remark}

The polynomials $\Upsilon_{K}$ giving the negative curves in the theorem are defined via the equation
\begin{align}\label{eqn-special}
  x(1-y)\Upsilon_{K}=\left(\xi_{2,K}^{int}\right)^{N} - \left(\xi_{1,K-1}^{int}\right)^{M+N}.
\end{align}
The polynomials $\xi_{2,K}^{int}$ and $\xi_{1,K-1}^{int}$ vanish to order $M+N$ and $N$ at $e$, respectively. It follows that, if it exists, $\Upsilon_{K}$ vanishes to order $m=N(M+N)-1$ at $e$.

To simplify the notation define the auxiliary polynomials
\[
  \Upsilon'_{K} := \left(\xi_{2,K}^{int}\right)^{N} - \left(\xi_{1,K-1}^{int}\right)^{M+N}.
\]
%
%Figure~\ref{both-triangles} shows the Newton polygon of both $\xi_{2,K}^{int}$ and $\xi_{1,K-1}^{int}$.
% \begin{figure}
%   \centering
%   \begin{minipage}{0.45\textwidth}
%     \centering
%     \includegraphics[width=0.9\textwidth]{example-image-a} % first figure itself
%   \end{minipage}\hfill
%   \begin{minipage}{0.45\textwidth}
%     \centering
%     \includegraphics[width=0.9\textwidth]{example-image-b} % second figure itself
%   \end{minipage}
%   \caption{Triangles $IT_{K-1}(N-1,1)$ and $IT_{K}(M,N)$.}
% \end{figure}

The outline of this section is as follows. We begin by showing that $x(1-y)$ divides $\Upsilon_{K}'$, so that the polynomials $\Upsilon_{K}$ are well defined. The irreducibility of $\Upsilon_{K}$ is a direct consequence of this. The remainder of the section is devoted to computing the deficiency of these curves.

Since $\Upsilon_{K}'=x\Upsilon_{K}-xy\Upsilon_{K}$, the Newton polygon of $xy\Upsilon_{K}$ can be obtained from that one of $\Upsilon_{K}'$ by removing the lowest lattice point on each column of its support. We use this in the proof of Proposition~\ref{se-prop-np} to compute the Newton polygon of $\Upsilon_{K}$. Having described the Newton polygon of $\Upsilon_{K}$, computing its deficiency is a direct application of Pick's theorem.

\begin{lemma}
  The polynomial $\Upsilon_{K}'$ is divisible by $x$ and $1-y$ for every $K\geq 4$.
\end{lemma}
\begin{proof}
  By definition, every $\xi^{int}$ is normalized to have constant term $1$, so clearly $x$ divides $\Upsilon_{K}'$. To see that $1-y$ divides $\Upsilon'_{K}$ note that $\xi_{1,K-1}^{int}|_{y=1} = (1-x)^{N}$ and $\xi_{2,K}^{int}|_{y=1} = (1-x)^{M+N}$. Indeed, the monomials $x^{N}y^{N+1}$ and $x^{M+N}y^{KN}$ are part of the support of $\xi_{1,K-1}^{int}$ and $\xi_{2,K}^{int}$, respectively, with coefficients $\pm 1$. Then, since $\xi_{2,K}^{int}\in (1-x,1-y)^{M+N}$ and $\xi_{1,K-1}^{int}\in (1-x,1-y)^{N}$, the result follows.
\end{proof}

\begin{proposition}\label{se-prop-irred}
  The polynomial $\Upsilon_{K}$ is irreducible.
\end{proposition}
\begin{proof}
  Let $\Delta OA'B'$ denote the triangle $N\cdot IT_{K}(M,N)$ and $\Delta OP'Q'$ the triangle $(M+N)\cdot IT_{K-1}(N-1,1)$. Here $O$ is the origin and vertices are written in counterclockwise order; see Figure~\ref{lemma-triangles}. 
  
  The coefficients of $\Upsilon_{K}'$ at $A'$ and $B'$ come uniquely from those of $\left(\xi_{2,K}^{int}\right)^{N}$. Indeed, LHS slope of $IT_{K}(M,N)$ is larger than that one of $IT_{K-1}(N-1,1)$, and the base of $N\cdot IT_{K}(M,N)$ is larger than that one of $(M+N)\cdot IT_{K-1}(N-1,1)$. By \cite[Paragraph~6.2]{GGK3}, the coefficients of $\xi_{2,K}^{int}$ at all of its vertices are $\pm1$, so the coefficients of $\Upsilon_{K}'$ at $A'$ and $B'$ are also $\pm 1$.

  Notice that $B'=Q'+(0,1)$, so that the coefficient at $Q'$ is the negative of that one at $B'$, this is, $\pm1$. Furthermore, by the same token the edge $\overline{A'Q'}$ must be an edge of the support of the $x\Upsilon_{K}$ term of $\Upsilon_{K}'$. Then, by a standard argument using Minkowski sums, $\Upsilon_{K}$ will be irreducible if we can prove that this edge has lattice length $1$, see for example \cite[Lemma~2.3]{GGK3}.

  The fact that $\overline{A'Q'}$ has length $1$ is immediate once we note that $\overline{A'B'}$ has integral slope $K$ and $B'=Q'+(0,1)$.
\end{proof}

The remainder of this section will be concerned with computing the deficiency of $\Upsilon_{K}$. For this we will describe the Newton polygon of $\Upsilon_{K}'$ and then use Pick's theorem to obtain the result.

\begin{lemma}\label{lemma-u-prime}
  After removing the origin, the convex hull of the remaining lattice points in the triangles $N\cdot IT_{K}(M,N)$ and $ (M+N)\cdot IT_{K-1}(N-1,1)$ has vertices $A',Q',B',C',D',(1,1)$ and $(1,0)$, where
  \begin{align*}
    A'&=\text{bottom RHS vertex of }N\cdot IT_{K}(M,N)=(MN,0),\\
    Q'&=\text{top vertex of }(M+N)\cdot IT_{K-1}(N-1,1)=(N(M+N),(N+1)(M+N)),\\
    B'&=\text{top vertex of }N\cdot IT_{K}(M,N)=(N(M+N),KN^{2}),\\
    C'&=\text{top vertex of }IT_{K}(M,N)=(M+N, KN)\text{ and}\\
    D'&=\text{top vertex of }IT_{K-1}(N-1,1)=(N,N+1).
  \end{align*}
  See Figure~\ref{lemma-triangles}. In particular, the polynomial $\Upsilon_{K}'$ is supported in this polygon.
\end{lemma}

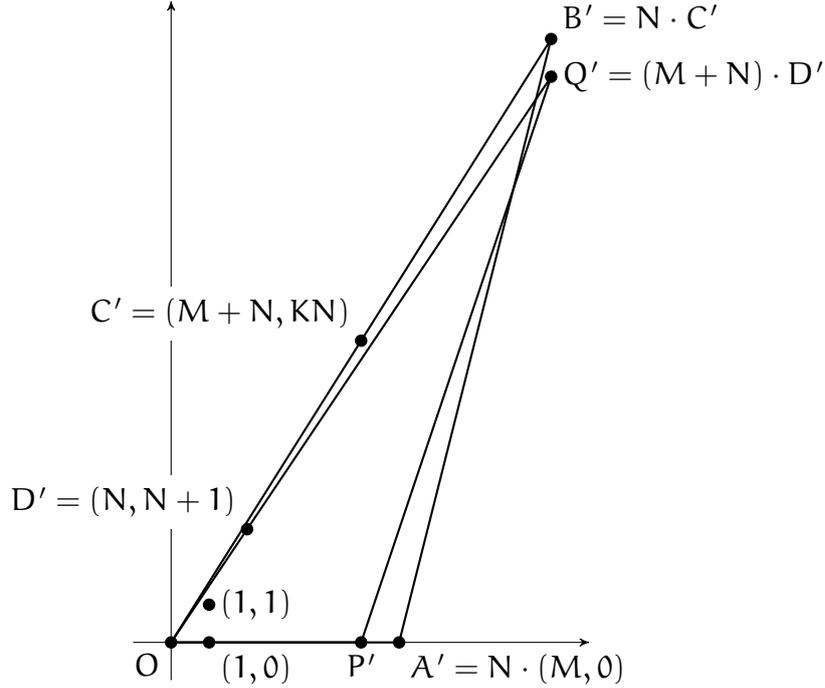
\begin{figure}[htb]
  \centering
  \begin{tikzpicture}[
  scale=0.5,
    axis/.style={ ->, >=stealth'},
    important line/.style={very thick},
    ]

    % x-axis
    \draw[axis] (-1,0) -- (11,0) node(xline)[right] {};    
    %y-axis
    \draw[axis] (0,-1) -- (0,17) node(yline)[above] {};

    \pgfmathsetmacro\K{4}
    \pgfmathsetmacro\N{int(\K - 2)}
    \pgfmathsetmacro\M{int(\N*\N -1)}
    
    \coordinate (O) at (0,0);
    \node at (O) [below left] {$O$};
    
    \coordinate (P) at ({(\M+\N)*(\N-1)},0);
    \node at (P) [below] {$P'$};
    \coordinate (Q) at ({(\M+\N)*\N},{(\M+\N)*(\N+1)});
    \node at (Q) [right] {$Q'=(M+N)\cdot D'$};
    
    \coordinate (A) at ({\N*\M},0);
    \node at (A) [below right] {$A'=N\cdot (M,0)$};
    \coordinate (B) at ({\N*(\M+\N)},{\K*\N^2});
    \node at (B) [above right] {$B'=N\cdot C'$};

    \coordinate (C) at ({\M+\N},{\K*\N});
    \node[rectangle, fill=white] at (C) [above left] {$C'=(M+N,KN)$};
    \coordinate (D) at (\N,{\N+1});
    \node[rectangle, fill=white] at (D) [above left] {$D'=(N,N+1)$};
    
    %\coordinate (F) at (\M,{\M+\N});
    %\node at (F) [below] {$base$};
    %-M
    
    \node at (1,0) [below right] {$(1,0)$};
    \node at (1,1) [right] {$(1,1)$};

    % Triangles
    \draw[thick] (O) -- (A) -- (B) -- cycle;
    \draw[thick] (O) -- (P) -- (Q) -- cycle;

    % Lattice points
     \foreach \Point in {
       (0,0),(1,1),(1,0),(A),(B),(P),(Q),(C),(D)%,(F),(F)+(D)
       %{($ (A)+(0,1) $)},  {($ (A) + (1,\K) $)},{($ (A) + 2*(1,\K) $)},{($ (A) + 3*(1,\K) $)},
       %{($ (A) + (1,\K+1) $)},{($ (A) + 2*(1,\K+1/2) $)},{($ (A) + 3*(1,\K+1/3) $)},
     }
     \draw[fill=black] \Point circle (0.15);

     % Convex hull
     %\draw[fill=red, opacity=0.3] (1,0) -- (1,1) -- (D) -- (C) -- (B) -- (Q) -- (A) -- cycle;
     %\draw[fill=red, opacity=0.2] (1,1) -- (D) -- (C) -- (B) -- ($(A)+(0,1)$) -- cycle;
\end{tikzpicture}

%%% Local Variables:
%%% mode: latex
%%% TeX-master: "se"
%%% End:
  \caption{Newton polygons of $\left(\xi_{1,K-1}^{int}\right)^{M+N}$ and $\left(\xi_{2,K}^{int}\right)^{N}$ on top of each other.}
  \label{lemma-triangles}
\end{figure}

\begin{proof}
  It follows from the discussion in the proof of Proposition~\ref{se-prop-irred} that the vertices $A',B'$ and $Q'$ all are all in the support of $\Upsilon_{K}'$ with coefficients $\pm 1$.

  Since the LHS slope of $IT_{K}(M,N)$ is greater than that one of $IT_{K-1}(N-1,1)$, the point $C'$ is also a vertex of the support of $\Upsilon_{K'}$ (one can check that it has coefficient $\pm N$ because the coefficient of $\xi_{2,K}^{int}$ at its top vertex is $\pm1$). It only remains to show that the missing vertices are $D',(1,1)$ and $(1,0)$. 
  
  To see this we show that the (non-convex) polygon with vertices $O,D',(1,1)$ and $(1,0)$ has no lattice points in its interior. Indeed, a direct computation shows that the vectors $(1,1)$ and $D'$ form a basis of the lattice, just as $D'$ and $C'$ do.%; see Figure~\ref{left-corner}.
\end{proof}

% \begin{figure}[htb]
%   \centering
%   \subcaptionbox{Each of the three pairs of consecutive vectors in counterclockwise order forms a basis of the lattice.\label{left-corner}}[.4\linewidth]{\input{left-corner}}
%   \hspace*{\fill}
%   \subcaptionbox{Lattice point distribution in $\Delta A'Q'B'$.\label{AQB}}[.4\linewidth]{\input{AQB}}
%   \caption{Lattice points on the left and right hand side edges of $OA'Q'B'$.}
% \end{figure}

With the previous result we can now determine the Newton polygon of $\Upsilon_{K}$.
%Propositions~\ref{prop-np} and \ref{prop-irred}.

\begin{proposition}\label{se-prop-np}
  The Newton polygon of $\Upsilon_{K}$ has vertices
  \begin{align*}
    O&=(0,0),\\
    A&=(MN-1,0),\\
    B&=(N(M+N)-1,KN^{2}-1),\\
    C&=(M+N-1, KN-1)\text{ and}\\
    D&=(N-1,N).
  \end{align*}
\end{proposition}

\begin{proof}
Since $\Upsilon'_{K} = x(1-y)\Upsilon_{K}$, the polynomial $\Upsilon_{K}$ is supported in the polygon obtained by removing the bottom layer of lattice points in the polygon from Lemma~\ref{lemma-u-prime} and then translating the result by $(-1,-1)$. We show that the resulting polygon has vertices $O,A,B,C$ and $D$ and that the coefficients of $\Upsilon_{K}$ at these points are nonzero.
  
First note that every lattice point along the bottom edge of the polygon supporting $\Upsilon_{K}'$ is removed by this process. All the other points that get removed are $Q'$ and all lattice points along the edge $\overline{A'B'}$ except for $B'$. Indeed, $\overline{A'B'}$ is the RHS edge of $N\cdot IT_{K}(M,N)$ and has integral slope, so the only possibility is for the points along it or below it to get removed by this process. However, since $B'=Q'+(0,1)$, it's clear that there are no lattice point points in the support of $\Upsilon_{K}'$ lying below those along the edge.%; see Figure~{\ref{AQB}}.
  
It follows that the Newton polygon of $xy\Upsilon_{K}$ is contained in the convex hull of $(1,1)$, $A'+(0,1)$, $B'$, $C'$ and $D'$. Translating by $(-1,-1)$ yields the polygon with vertices $O,A,B,C$ and $D$.

Now we show that $\Upsilon_{K}$ has nonzero coefficients at these points.

From the proof of Lemma~\ref{lemma-u-prime} we know that the coefficients of $\Upsilon_{K}$ at $A$ and $B$ are $\pm 1$. Similarly, the coefficient of $\Upsilon_{K}'$ at $C'$ is $\pm N$ because the coefficient of $\xi_{2,K}^{int}$ at its top vertex is $\pm1$, so the same is true for the coefficient of $\Upsilon_{K}$ at $C$.

The constant term of $\Upsilon_{K}$ is equal to $\pm1$. It's nonzero because otherwise the polynomial would be divisible by $1-y$ by Lemma~\ref{lemma-width}. This is impossible because $B$ is in its support with coefficient $\pm1$ and no other point along the same column is. Furthermore, the coefficient is equal to $\pm1$ because the previous argument as well as the proof of Lemma~\ref{lemma-width} are independent of the characteristic of the field.

To finish the proof we show that the coefficient of $\Upsilon_{K}$ at $D$ is nonzero or, equivalently, that the same is true for the coefficient of $\Upsilon_{K}'$ at $D'$. The contributions to this coefficient come from both $\left(\xi_{2,K}^{int}\right)^{N}$ and $\left(\xi_{1,K-1}^{int}\right)^{M+N}$. On the one hand, $D'$ is the top vertex of $IT_{K-1}(N-1,1)$, so $\left(\xi_{1,K-1}^{int}\right)^{M+N}$ contributes by $\pm (M+N)$. On the other hand, we will show that $\left(\xi_{2,K}^{int}\right)^{N}$ has coefficient $rN$ at this vertex for some $r\in\ZZ$. This would imply the claim because the coefficient of $\Upsilon_{K}$ at $D$ would be, up to a sign,
\[
  (M+N) \pm rN =  M + (1\pm r)N\neq 0.
\]
This can't be zero because $M$ and $N$ are relatively prime.

To compute the contribution of $\left(\xi_{2,K}^{int}\right)^{N}$ to this coefficient note that the segment $\overline{OD'}$ has length $1$, and the only other lattice point in $IT_{K}(M,N)$ having slope greater than or equal to that one of $D'$ is its top vertex. However, the top vertex is further away from the origin than $D'$, so the coefficient of $\left(\xi_{2,K}^{int}\right)^{N}$ at $D'$ is the same as that one of $(1+ r\chi^{D'})^{N}$, where $r$ is the coefficient of $\xi_{2,K}^{int}$ at $D'$. Since the polynomials $\xi_{IT}$ are defined over $\ZZ$, this concludes the proof.
\end{proof}

Computing the deficiency of $\Upsilon_{K}$ is now a straightforward but lengthy computation, which is left to the reader. We present the following lemma as a guide.

\begin{lemma}
  The Newton polygon of $\Upsilon_{K}$ has:
  \begin{enumerate}[(i)]
  \item Lattice perimeter $N(M+1)+1$, and
  \item $\operatorname{Area}=\frac{1}{2}\left(K M N^3 - K N^2 - M N + M + N\right)$.
  \end{enumerate}
  By Pick's theorem, its number of lattice points is:
  \[
    \frac{1}{2}\left(K M N^3 - K N^2 +2N + M + 3\right).
  \]
  As a consequence, the deficiency of $\Upsilon_{K}$ is $\frac{(K-2)(K-3)}{2}$.
\end{lemma}
\begin{proof}
  We show only (i). Let $A,B,C$ and $D$ be as in Proposition~\ref{se-prop-np}. 
  \begin{enumerate}[(i)]
  \item From the proof of Lemma~\ref{lemma-u-prime} we know that the segments $\overline{OD}$ and $\overline{CD}$ have lattice length $1$ and $\overline{OA}$ has lattice length $MN-1$. Similarly,  we showed above that the edge $\overline{AB}$ has lattice length $1$. Finally,
    \[
      B = C + (N-1)(M+N,KN),
    \]
    where $M+N$ and $KN$ are relatively prime. Therefore, $\overline{BC}$ has length $N-1$. Adding up these lengths yields a perimeter of $N(M+1)+1$.
  \end{enumerate}
\end{proof}

To conclude this section we piece together all the previous result to prove Theorem~\ref{thm-special}.

\begin{proof}[Proof of Theorem~\ref{thm-special}.]
  By equation~\ref{eqn-special} we know that $\Upsilon_{K}$ has $m=N(M+N)-1$. The previous lemma shows that this curve has the desired deficiency. It only remains to show that the polynomial is supported in the triangle described in the statement. The vertices $A$ and $B$ from Proposition~\ref{se-prop-np} are precisely the ones described in the statement of the theorem, and a direct computation shows that the edge $\overline{BC}$ has slope $\frac{KN}{M+N}$.

  The self-intersection of the negative curve $C$ defined by $\Upsilon_{K}$ is:
  \[
    C\cdot C = \left( MN-1 + \frac{K-1}{KN}\right)(KN^{2} -1) - (N(M+N)-1)^{2}= -\frac{K-1}{KN}.
  \]
\end{proof}

%%% Local Variables:
%%% mode: latex
%%% TeX-master: "ms"
%%% End:

\section{The Mori Dream Space property}

In this section we consider the Mori Dream Space property of the blowups $X$. The main observation is that this property is often determined by the location of the diamonds of the negative curves on the $st$-plane.

By a result of Cutkosky, $X$ is a MDS if and only if it has a negative curve $C$ and an irreducible curve $D$ disjoint from it. Consider a blowup $X$ that corresponds to a point on the $st$-plane where two diamonds meet. The two negative curves are disjoint in $X$, hence $X$ is a MDS. If we now move to the interior of one diamond, then under some conditions the two curves stay disjoint and the variety $X$ stays a MDS. In fact, we will see that this idea explains almost all known cases of the MDS property for $X$. We start with a proposition that is the basis for this argument.

We say that a triangle $\Delta$ supports a negative curve $C$ if the Newton polygon of the equation of $C$ lies in $\Delta$, and $\Delta$ is the smallest such triangle parallel to $\Delta$.

\begin{proposition} \label{prop-CD}
  Let $\Delta_1 \subseteq \Delta_2$ be two triangles that support the same negative curve $C$. If $X_2=\Bl_e X_{\Delta_2}$ is a MDS then so is $X_1=\Bl_e X_{\Delta_1}$.
\end{proposition}

\begin{proof}
  Consider a variety $Y$ with proper birational morphisms $\pi_i: Y\to X_i$, $i=1,2$. Here $Y=\Bl_e Z$, where $Z$ is the toric surface whose fan is the coarsest refinement of the normal fans of $\Delta_1,\Delta_2$. We consider curves in $Y$ and $X_i$ as either $\QQ$-Cartier divisors or as Chow classes. We then have the pullback $\pi_i^*$ of divisors and the proper pushforward $(\pi_i)_*$ of the Chow classes.

  Let $C_i\subset X_i$ be the negative curves. The assumption that $\Delta_1\subseteq \Delta_2$ means that $\pi_1^*(C_1)$ is an irreducible curve in $Y$. Indeed, the smallest polygon that supports the equation of $C_1$ and has sides parallel to the sides of $\Delta_1$ and $\Delta_2$ is the triangle $\Delta_1$. Then $C_2= (\pi_2)_* \pi_1^* (C_1)$.

  Let $D_2$ be an irreducible curve in $X_2$ orthogonal to $C_2$, and let $D_1= (\pi_1)_* \pi_2^* (D_2)$. This $D_1$ is an effective curve in $X_1$, possibly reducible. It is orthogonal to $C_1$:
  \[ C_1 \cdot D_1 = C_1 \cdot (\pi_1)_* \pi_2^* (D_2) = \pi_1^*(C_1) \cdot \pi_2^* (D_2) = (\pi_2)_* \pi_1^* (C_1) \cdot D_2 = C_2 \cdot D_2 = 0.\]
  The curve $D_1$ consists of the strict transform of the curve $D_2$ plus a linear combination of boundary divisors. There is no component $C_1$ in it, hence $D_1$ and $C_1$ are disjoint. This implies that $D_1$ is in fact irreducible, equal to the strict transform of $D_2$.
\end{proof}

The proposition shows that if the triangles supporting the negative curve $C$ satisfy the containment $\Delta_1 \subseteq \Delta_2$, then the equation defining the curve $D$ in $X_2$ also defines the curve $D$ in $X_1$. This implies that the triangles supporting the equation of $D$ satisfy the reverse containment $\Delta'_2 \subseteq \Delta'_1$.

Next we explain how to check if $\Delta_1 \subseteq \Delta_2$. Let $(s_i,t_i)$ be the slopes of $\Delta_i$. We start with the cases where either $s_1=s_2$ or $t_1=t_2$. These are determined as follows:
\begin{enumerate}
\item Assume that $\Delta_1$ has the lower left vertex integral. If $s_2>s_1, t_2=t_1,$ then $\Delta_1 \subseteq \Delta_2$. Indeed, we pivot the left edge of $\Delta_1$ around its lower left vertex, making the triangle larger.
\item Assume that $\Delta_1$ has the lower right vertex integral. If $s_2=s_1, t_2<t_1,$ then $\Delta_1 \subseteq \Delta_2$.
\item Assume that $\Delta_1$ has the top vertex integral. If $s_2=s_1, t_2>t_1,$ or $s_2<s_1, t_2=t_1,$ then $\Delta_1 \subseteq \Delta_2$.
\end{enumerate}
All other cases do not yield $\Delta_1 \subseteq \Delta_2$ when varying  $s_1$ or $t_1$.
We can combine these cases and vary both $s_1$ and $t_1$:
\begin{enumerate}
\item Assume that $\Delta_1$ has the top vertex integral. If $s_2\leq s_1, t_2 \geq t_1,$ then $\Delta_1 \subseteq \Delta_2$. We first change $s_1$ to $s_2$ by pivoting around the top vertex. This keeps the top vertex integral, so we can change $t_1$ to $t_2$. Similarly, if $\Delta_1$ has its lower left and right vertices integral and $s_2\geq s_1, t_2 \leq t_1,$ then $\Delta_1 \subseteq \Delta_2$. 
\item Consider the case $s_2 > s_1, t_2 > t_1$. Even if all vertices of $\Delta_1$ are integral, the condition $\Delta_1 \subseteq \Delta_2$ may not be satisfied. Indeed, if we first change $s_1$ to $s_2$, then the top vertex may become rational, and we are not allowed to change $t_1$ to $t_2$. 
\end{enumerate}

We now consider various diamonds of negative curves and show that the previous proposition explains most known examples where the MDS property holds, at least for the case $m>1$. 
\begin{enumerate}
\item Consider the triangle $IT(M,N)$. It supports a negative curve with Newton polygon equal to the triangle itself. Let $(s_0,t_0)$ be its coordinates. This is the point in the centre of the diamond. Consider the line segment from $(s_0,t_0)$ to $(s_0,t_1)$, where the second point lies at the top vertex of the $IT(M,N)$ diamond. The point $(s_0,t_1)$ is the intersection point of the $IT(M,N)$ diamond and the region of the negative curve defined by $1-y$. As we move up from $(s_0,t_0)$ towards $(s_0,t_1)$, we increase the triangle supporting the negative curve $C$. Since the variety corresponding to $(s_0,t_1)$ is a MDS, so is every other variety $X$ corresponding to a point on the line segment.

  The left and right vertices of the $IT(M_n,N_n)$ diamond meet diamonds of $RT(M_{n-1},N_{n-1})$ and $RT(M_n,N_n)$. By a similar argument, the whole horizontal line in the $IT(M,N)$ diamond through $(s_0,t_0)$ has the MDS property. The lower vertex of the $IT(M,N)$ diamond meets another diamond only for $(M,N)=(M_n,N_n)$, $n=1,2$ and for $(M,N)=(N_1,M_1)$. For instance, the diamond of $IT_K(M_2,N_2)=IT_K((K-2)^2-1, K-2)$ meets the diamond of the special negative curve in the infinite family. In these cases the lower half of the vertical axis in the $IT(M,N)$ diamond has the MDS property.

\item For triangles $RT(M,N)$, the same argument as in the integral case shows that points on the horizontal axis and the upper half of the vertical axis have the $MDS$ property. When we move down from $(s_0,t_0)$, then the triangle supporting $C$ does not get bigger, hence we cannot apply the proposition to this case. Except for the small values $(M_1, N_1)=(K-2, 1)$ and $(N_1, M_1)$, the lower vertex of the $RT(M,N)$ diamond is irrational, hence one does not expect it to meet another diamond.

\item Consider the diamond of the negative curve that is supported in $RT_4(2,1)$ and $IT_3(1,1)$ (see Figure~\ref{fig-m2}). This diamond meets other diamonds not only at its four vertices but also at other points on its boundary. Let $(s_0,t_0) = (1.5, 3)$ be the centre of the diamond and $(s_1, t_1)$ a point on the boundary that also lies in another diamond. If the point $(s_1,t_1)$ lies either on the north-west or south-east boundary, then all points in the rectangle $[s_{0}, s_{1}]\times[t_{0},t_{1}]$ have the MDS property. Indeed, for all these points the triangle supporting the negative curve $C$ lies in the triangle at $(s_1,t_1)$. If the $(s_1, t_1)$ lies on the north-east boundary, then we can only conclude the MDS property for the horizontal line $s_{0} \leq s \leq s_{1},t=t_{1}$. Along this line we have containment of triangles. Similarly, when $(s_1,t_1)$ lies on the south-west boundary, we get the MDS property on the vertical line $s=s_1, t_{1} \leq t \leq t_{0}$. 

  The triangle $RT_4(2,1)$ contains the same negative curve, but has coordinates $(1.5, 4)$. The general proof for triangles $RT(M,N)$ in this case shows that points above the line $t=4$ are non-MDS, unless they lie on the vertical axis $s=1.5$. 

  There is an obvious symmetry in the diamond. This stems from the fact that only the north-east half of the diamond lies in the fundamental domain. There is an isomorphism of triangles that maps the diamond to itself. It is given by first transforming the triangle by the shear transformation $(x,y)\mapsto (x-y,y)$ and then reflecting in the $y$ axis. 

  A similar argument applies to the negative curves supported in $RT(K-2, 1)$ and $IT(K-3,1)$ for $K>4$. 

  \begin{figure}
    \centering
    \includegraphics[width=\textwidth]{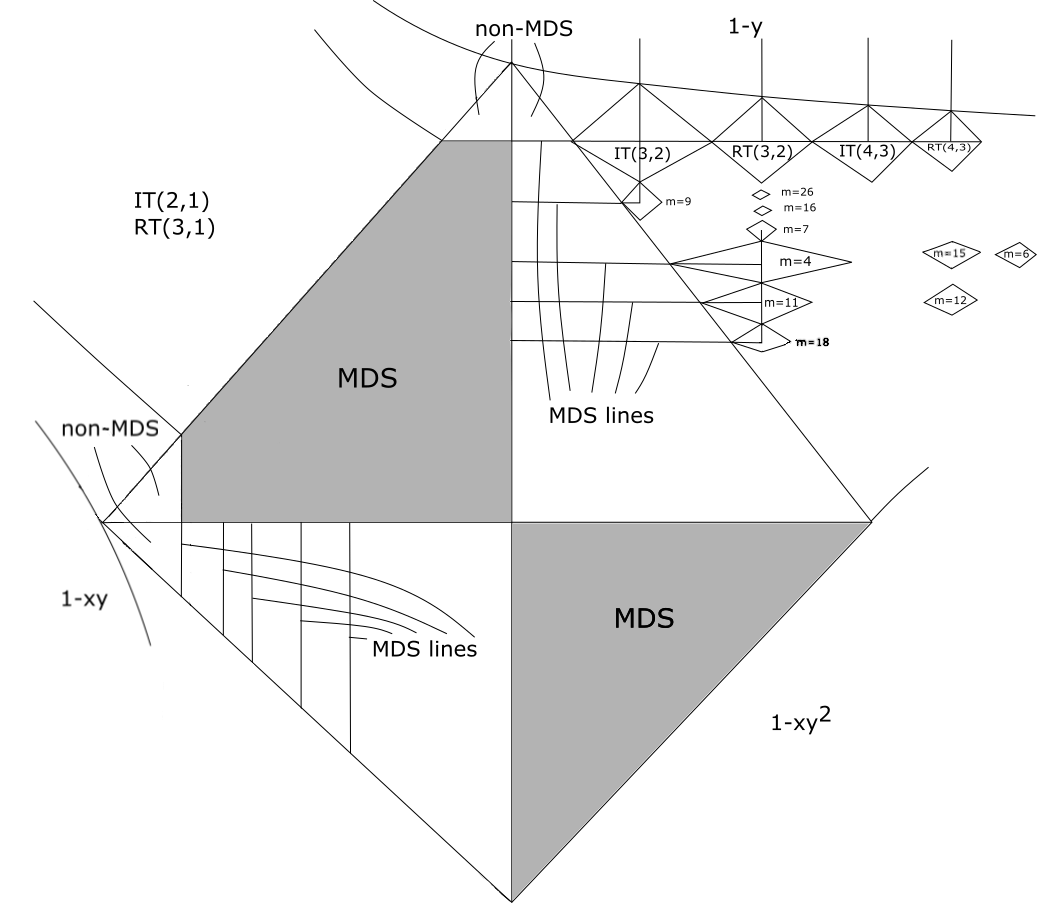}
    \caption{The big diamond of the negative curve in $IT_3(1,1)$ and $RT_4(2,1)$ with $m=2$. Figure is not to scale.}
    \label{fig-m2}
  \end{figure}

\item Consider the new family of non-special curves constructed in Section~\ref{sec-new-nonspecial}. For even $K$, the diamonds of these curves meet the diamonds of $RT_K(M_1,N_1)$. Hence, the left half of their horizontal axis has the MDS property. In particular, the centres of the diamonds, corresponding to the minimal triangles that support the negative curve, have the MDS property.

\item The diamonds of the infinite family of special curves meets the diamond of $IT_K(M_2,N_2)$ at its highest point and the diamond of $RT_K(M_1,N_1)$ at its left vertex. This means that the top half of the vertical axis and the left half of the horizontal axis have the MDS property. In particular, the centre points $(s_0,t_0)$ have the MDS property.

  The $m=18$ special curve from \cite{KuranoMatsuoka} is similar to the other special curves. Its diamond meets the $RT(2,1)$ diamond at its left vertex and the $m=11$ diamond at the top vertex.  In particular, the blowups of the two examples in \cite{KuranoMatsuoka}, $\PP(5, 33, 49)$ and $\PP(8,15,43)$ are both MDS.

\item In the column $s=5/3$ there is a continuous string of diamonds corresponding to curves with $m=7,4,11,18$. We get the MDS property in the line segment from the centre of the $m=7$ diamond to the centre of the $m=18$ diamond. This gives a range of $b,c$ such that $\Bl_e \PP(5,b,c)$ is a MDS.

\item Consider the region where the negative curve is defined by the vanishing of $1-y$. Here if we have a point $(s,t)$ with the MDS property, then any point above it, $(s,t')$ with $t'>t$, also has this property. We will see in the following lemma that all points that lie on the horizontal lines having $t=K$ an integer are MDS. Other than these examples, we also get vertical lines above the centres of $IT(M,N)$ and $RT(M,N)$.
\end{enumerate}

\begin{lemma} 
  Let $\Delta$ support the negative curve $C$ defined by the polynomial $1-y$. If the right slope $t$ of $\Delta$ is an integer, then $X=\Bl_e X_\Delta$ is a MDS.
\end{lemma}

\begin{proof}
  The class of $D$ that is orthogonal to $C$ must have the form
  \[[D] = H' -nE,\]
  where $H'$ corresponds to a triangle $\Delta'$ parallel to $\Delta$ and with width equal to $n$. We take for $\Delta'$ the smallest such triangle that has the left edge with integral vertices and claim that this triangle supports the required curve $D$.

  To prove the claim, we use the same argument as in \cite[Section~2]{GGK3}. Notice that $\Delta'$ contains $n+1$ columns of lattice points. If the count of lattice points in columns is $1,1,2,3,4,\ldots, n$, then $\Delta'$ supports a unique polynomial with vanishing order $n$ at $e$, and whose coefficients at the left and top lattice points are nonzero. This polynomial then defines the curve $D$ that is disjoint from $C$. We will show that this count of lattice points in columns holds after we remove some lattice points in $\Delta'$. The remaining points then support the polynomial defining $D$.

  The vertical line through the lower right vertex divides $\Delta'$ into two sub-triangles. We transform $\Delta'$ by applying a shear transformation to the right sub-triangle, making the right edge with slope $t$ horizontal. This operation does not change the number of lattice points in columns. The result of this shear is a new triangle with horizontal base $[0,n]$, and height $\geq n$. Indeed, if $C^2\leq 0$ then the area of $\Delta'$ is $\geq \frac{n^2}{2}$. Choose a sub-triangle $\Delta''$ in the sheared triangle with the same base $[0,n]$ and height equal to $n$. We consider all lattice points in $\Delta''$ except those on the left edge other than the left vertex. These lattice points have as their row count $n+1,n-1,n-2,\ldots,3,2,1$, and hence they have the correct count in columns. 
\end{proof}

Notice that in the case where the curve $C$ is defined by $1-y$, the curve $C$ passes through the torus-fixed points corresponding to left and top vertices of the triangle. The curve $D$ must then have these vertices in its Newton polygon. This means that if the two triangles have different slopes $s$ of the left edge then the curves $D$ must be defined by different polynomials. Hence, for every rational point in the line $t=K$ we get a different polynomial defining the curve $D$.

The case where the curve $C$ is defined by $1-y$ is also interesting for the following reason. In all other examples listed above, the curve $D$ is defined by a polynomial that also defines a negative curve in a neighbouring diamond. This is not necessarily true in the $1-y$ case. 

\subsection{The non-MDS property}

Proving that a variety $X$ is not a MDS is more difficult than proving that it is a MDS. We expect that all $X$ are non-MDS unless there is a clear reason for them to be a MDS, like the ones listed above. We prove this for the diamonds of $IT(M,N)$ and the upper half of the diamonds of $RT(M,N)$, except for some small values of $M$ and $N$.

\begin{lemma}
  The following points in the diamonds of $IT_K(M,N)$ and $RT_K(M,N)$ have the non-MDS property:
  \begin{enumerate}
  \item When $N>K-2$, all points in the diamond of $IT_K(M,N)$ except those on the horizontal axis and the upper half of the vertical axis.
  \item When $N=K-2$ for $IT_K(M,N)$ and $M+N > 1$ for $RT_K(M,N)$, all points in the upper part of the diamond, $t>K$, except those on the vertical axis.
  \end{enumerate}
\end{lemma}

\begin{proof}
  These results were proved in \cite[Lemma 5.5]{GGK3} for the upper left quadrant of the diamonds. The same proof works also for the upper right quadrant. We will recall the main steps of the proof and then use a similar argument for the lower half of the diamond in the integral case.

  Let $(s_0,t_0=K)$ be the coordinates of the triangle $IT_K(M,N)$ or $RT_K(M,N)$, and let $(s_1,t_1>K)$ be a point in the upper half of the diamond. We let the corresponding triangles that support the same negative curve $C$ be $\Delta_0$ an $\Delta_1$. Then $\Delta_0 \subseteq \Delta_1$. The proof of Proposition~\ref{prop-CD} shows that if a polynomial $\zeta$ defines the curve $D$ for $\Delta_1$, then the same polynomial also defines the curve $D$ for $\Delta_0$. However, we already know several curves $D$ for $\Delta_0$, for example the curve defined by $1-y$. For suitable integers $q,r$, the polynomials $\zeta$ and $x^q(1-y)^r$ lie in the same class and hence are sections of the same line bundle. Since these sections do not vanish when restricted to the curve $C \subseteq \Bl_e X_{\Delta_0}$, one is a constant multiple of the other when restricted to $C$. Let $\xi$ be the polynomial defining $C$. Then we can write 
  \[ \zeta = cx^q(1-y)^r + \xi g\]
  for a suitable constant $c$ and polynomial $g$.    
  The contradiction to the existence of $\zeta$ comes from checking the class of the curve $\overline{D}$ defined by $g$.  
  Assuming that $t_1>t_0$, this curve $\overline{D}$ intersects $C$ negatively in $\Bl_e X_{\Delta_0}$ by the argument in \cite[Lemma 5.3 and Lemma 5.5]{GGK3}, which is stated for the upper left quadrant of the diamond but runs word-by-word also for the upper right quadrant. 
  Hence, $D$ must have $C$ as a component, which means that $\xi$ divides $g$. However, this can be ruled out by restricting the polynomials to the left edge of the triangle $\Delta_0$, and showing that the restriction of $\zeta - cx^q(1-y)^r$ is not divisible by the restriction of $\xi^2$. 
  Indeed, assuming that $s_1\neq s_0$, $\zeta - cx^q(1-y)^r$ is supported at two lattice points on the left edge. The multiple roots of a single variable polynomial that contains exactly two nonzero terms must all be zero, but the roots of $\xi$ are not all zero because $\xi$ also has two nonzero terms on the left edge. 
   
%  The conclusion in the case of one point is clear and in the case of two points it also holds because the restriction 
% of $\zeta - cx^q(1-y)^n$ considered as a one variable polynomial cannot have a repeated nonzero root because % it is supported at exactly two points.  

  We will now use a similar argument for the lower half of the $IT(M,N)$ diamonds. We let $\Delta_0, \Delta_1$ be as before, with $(s_1,t_1<K)$ now in the lower half of the diamond. Then in the integral case $\Delta_0\subseteq \Delta_1$. (This is no longer true for rational triangles and that is why we cannot say anything about the lower half of these diamonds.) Assuming that there is a curve $D$ for $\Delta_1$ defined by a polynomial $\zeta$, we get two such curves $D$ for the triangle $\Delta_0$, defined by $\zeta$ and $1-y$. As before, these must satisfy an equation
  \[ \zeta = cx^q (1-y)^r + \xi g\]
  for a constant $c$ and a polynomial $g$. 
  We can assume that $\zeta$ has rational coefficients, since defining a (possibly reducible) $D$ curve is characterized by the coefficients of $\zeta$ satisfying a system of homogeneous linear equations with rational coefficients, so if there is a nonzero solution over $k$ there is one over $\mathbb{Q}$.  Since the monomial $x^q$ has a zero coefficient in $\xi g$, we deduce that in this case $c \in \mathbb{Q}$. 
  We get a contradiction by showing that the difference $\zeta - cx^q(1-y)^r$ is not divisible by $\xi$. This can be detected already when we restrict the polynomials to the right edge of the triangle $\Delta_0$. These restrictions lie in the ring $\mathbb{Q}[z]$, where we let the lower right vertex correspond to $1$, the next lattice point to $z$, and so on. The restriction of $\zeta - cx^q (1-y)^r$ has the form $dz^p - c$ for some integer $p>0$ and some $d \in \mathbb{Q}$. Notice that all roots of this polynomial have the same absolute value. Now it suffices to show that $\xi$ has roots with different absolute values.

  To help the flow of this argument we postpone some computational details to Lemma~\ref{lemma.same.norm} and Lemma~\ref{lemma.bound.an} below. Let $(M_{n},N_{n})=(F_{n+1},F_{n})$ be as defined in Subsection~\ref{subsection-IT-RT}. We will show in the proof of Lemma~\ref{lemma.bound.an} that the polynomial $\xi=\xi^{int}_{M,N}$, when restricted to the right edge of the triangle, has the form 
  \[
    \pm 1 + a_nz + \ldots + b_n z^{N-1} \pm z^N.
  \] 
  By Lemma~\ref{lemma.same.norm}, if all (complex) roots of $\xi$ have the same absolute value, then $|a_n|=|b_n|$. The coefficient $b_n$ was computed in \cite[Lemma 6.3]{GGK3} and we will bound the coefficient $a_n$ in Lemma~\ref{lemma.bound.an}. Then, these two coefficients are as follows.

  First assume that $M>N$, so that $(M,N) = (M_n,N_n)=(F_{n+1}, F_n)$. Then,
  \[
    |b_n| = F_{n-1}+F_{n-2}, \quad |a_{2n}| = |a_{2n-1}| \leq n.
  \]
  For $K\geq 4$ and $n\geq 3$ this implies $|b_n|>|a_n|$. 

  The case $M<N$ corresponds to $(M,N) = (N_n, M_n)=(F_{n}, F_{n+1})$. In this case
  \[
    |b_n| = F_{n}+F_{n-1}, \quad |a_{2n+1}| = |a_{2n}| \leq n+1.
  \]
  Now, for $K\geq 4$ and $n\geq 2$ we have $|b_n|>|a_n|$.
  % To help the flow of this argument we moved some computational details to Lemma~\ref{lemma.same.norm} and Lemma~\ref{lemma.bound.an} below. We will show in the proof of Lemma~\ref{lemma.bound.an} that the polynomial $\xi$ when restricted to the right edge of the triangle has the form 
  % \[ \xi = \pm 1 + a_nz + \ldots + b_n z^{N-1} \pm z^N.\] 
  % By Lemma~\ref{lemma.same.norm}, if all (complex) roots of $\xi$ have the same absolute value, then $|a_n|=|b_n|$. The coefficient $b_n$ was computed in \cite[Lemma 6.3]{GGK3} and we will bound the coefficient $a_n$ in Lemma~\ref{lemma.bound.an}. Then, these two coefficients are as follows.
  %
  % First assume that $M>N$, $(M,N) = (M_n,N_n)=(F_{n+1}, F_n)$. Then 
  % \[ |b_n| = F_{n-1}+F_{n-2}, \quad |a_{2n}| = |a_{2n-1}| \leq n.\]
  % Now for $K\geq 4$ and $n\geq 3$ we have $|b_n|>|a_n|$. 
  %
  % When $M<N$, $(M,N) = (N_n, M_n)=(F_{n}, F_{n+1})$, then 
  % \[ |b_n| = F_{n}+F_{n-1}, \quad |a_{2n+1}| = |a_{2n}| \leq n+1.\]
  % Now for $K\geq 4$ and $n\geq 2$ we have $|b_n|>|a_n|$.
\end{proof} 

To summarize, the only cases where we don't know about the MDS property for $IT_K(M,N)$ and $RT_K(M,N)$ are (assume $m>1$):
\begin{enumerate}
\item The lower half, $t< K$, of the diamonds of $RT_{K}(M,N)$.
\item The lower half of the diamonds of the first integral triangles $IT_K(M_2, N_2) = IT_K((K-2)^2-1, K-2)$ and $IT_K(N_1, M_1)= IT_K(K-2, 1)$. For example, $IT_4(3,2)$, $IT_4(1,2)$. We know that at least the vertical axis in these diamonds consists of MDS points. 
\item The diamonds of the first rational triangles $RT_K(M_1, N_1) = RT_K(K-2,1)$. (These are the same as the diamonds $IT_{K-1}(M_1,N_1)$). For example, the diamond of $RT_4(2,1)$ and $IT_3(1,1)$ that is shown in Figure~\ref{fig-m2}.
\end{enumerate}

\begin{lemma}    \label{lemma.same.norm}
  Consider a real polynomial of the form 
  \[ \xi = \pm 1 + az + \ldots + b z^{N-1} \pm z^N.\]
  If all roots of the polynomial have the same absolute value then $|a|=|b|$.
\end{lemma}

\begin{proof}
  Since the highest and lowest degree terms have coefficients $\pm 1$, if all roots have the same absolute value, this absolute value must be $1$. Let $\alpha_i$ be the (complex) roots of the polynomial. Then
  \[ |a| = \left|\sum \frac{1}{\alpha_i}\right| = \left|\sum \overline{\alpha_i}\right| = \left|\sum \alpha_i\right| = |b|.\]
\end{proof} 

Let us now bound the coefficient $a$ in $\xi^{int}$. Fix $K\geq 4$ and consider $M>N$, $(M,N) = (M_n,N_n)$. We are interested in the polynomial  $\xi_n^{int}$ restricted to the right edge of the triangle. The isomorphism of triangles $IT(M,N)\isom IT(N,M)$ in \cite[Lemma 3.3]{GGK3} exchanges the lower edges. Hence, $\xi$ restricted to the right edge in the case $M<N$ is equal to $\xi$ restricted to the bottom edge in the case $M>N$. To cover both cases, it suffices to study the coefficients of $\xi_n^{int}$ restricted to the lower two edges. 

Denote the terms in $\xi_n^{int}$ supported on the lower two edges of the triangle as
\[ \xi_n^{int} = 1+ \ldots + a'_n x^{M-1} + \delta_n x^M + a_n x^{M+1} y^K + \cdots \pm x^{M+N} y^{NK}.\]
If we restrict this polynomial to the right edge we get $\delta_n + a_n z +\ldots \pm z^N$, where we set $z^j=x^{M+j} y^{jK}$. If we restrict the polynomial to the base of the triangle we get $z^M+ \cdots + a'_n z + \delta_n$, where we set $z^j = x^{M-j}$. 
The coefficients $a_n$ and $a_n'$ are the ones we need to bound.

\begin{lemma}      \label{lemma.bound.an}
  With notation as above, 
  \[|a_{2n}| = |a_{2n-1}| \leq n, \quad |a'_{2n+1}| = |a'_{2n}| \leq n+1. \]
\end{lemma}

\begin{proof}
  We will use the recurrence relations (\ref{eqn.recurrence.relations}). Let us start by showing that the coefficient of $x^M$ in $\xi_n^{int}$ (corresponding to the right vertex) is $\delta_n=\pm 1$. In the first relation 
  \[ \xi^{int}_{n} = \xi^{rat}_{n} \xi^{rat}_{n-1} -\varepsilon^{rat}_{n} x^{M_n}(y-1)^{M_n+N_n},\]
  the terms $\xi^{rat}_{n} \xi^{rat}_{n-1}$ do not involve the right vertex of the triangle. Hence, the coefficient of $x^M$ is 
  \[ -\varepsilon^{rat}_{n} (-1)^{M_n+N_n} = \pm 1.\]

  Let us now consider restriction of the polynomials to the right edge of the triangle. These restrictions will be in the ring $k[z]$, where the constant $1$ corresponds to the right vertex and $z$ corresponds to the next lattice point $(M+1, K)$ on the edge. Then we may write the restriction of $\xi^{int}_{n}$ as 
  \[ \pm 1 + a_nz + \cdots \pm z^{N}.\]
  We will also need the restrictions of $\xi^{rat}_{n}$ to the right edge of their triangles. Note that these triangles have the right vertex rational, hence there is no constant term in the polynomials.  
 By \cite[Lemma 3.1]{GGK3}, for $n$ even the first lattice point has $y$-coordinate $1$, and for $n$ odd the first lattice point has $y$-coordinate $K-1$. 
 We will write the restriction of $\xi^{rat}_{n}$ as $c_n z^{\frac{1}{K}} + \cdots$ in the case of even $n$ and $ c_n z^{\frac{K-1}{K}} + \cdots$ in the case of odd $n$. 

  Considering the coefficients of $z$ in the two recurrence relations in (\ref{eqn.recurrence.relations}), we get
  \begin{align*}
    \delta_n a_{n-1} + \delta_{n-1} a_n &= \begin{cases} c_{n-1}^K & \text{if $n$ is odd,} \\ 0 & \text{if $n$ is even,} 
    \end{cases}\\
    c_n c_{n-1} &= a_n.
  \end{align*}
  The two cases in the first equation come from the fact that $(\xi^{rat}_{n-1})^K$ involves the monomial $z$ only in the case where $n$ is odd. The initial values are $a_0=0$, $c_0=-1$.

  The first equation implies that $a_{2n} = \pm a_{2n-1}$. From the second equation we can expand
  \[ c_n = \frac{a_n}{c_{n-1}} = \frac{a_n c_{n-2}}{ a_{n-1}} = \cdots.\]
  Using the equality $a_{2n} = \pm a_{2n-1}$ from the first equation and the initial value $c_0=-1$, we see that 
  \[ |c_{2n}| = \left|\frac{a_{2n}}{a_{2n-1}}\right| \cdots \left|\frac{a_2}{a_1}\right| = 1\quad\text{and}\quad |c_{2n+1}| = | a_{2n+1}|.\] 
  The inequality from the first relation
  \[ |a_{2n+1}| = | \pm a_{2n} \pm c_{2n}^K| = | \pm a_{2n} \pm 1| \leq |a_{2n}| + 1\]
  now proves the claim for $a_n$. 

  The proof for $a'_{n}$ is very similar. We now restrict the polynomials to the base of the triangles. These restrictions lie in $k[z]$ where we denote by $1$ the monomial corresponding to the right vertex and by $z$ the monomial corresponding to the next lattice point. Then the restriction of $\xi^{int}_{n}$ is
  \[ \pm 1 + a'_nz + \cdots + z^{M}.\]
  The restriction of $\xi^{rat}_{n}$ is $c'_n z^{\frac{K-1}{K}} + \cdots$ in the case of even $n$ and $ c'_n z^{\frac{1}{K}} + \cdots$ in the case of odd $n$. (This is because in the even case the first lattice point on the bottom edge is distance $(K-1)/K$ from the right vertex in the even case, and distance $1/K$ in the odd case.)

  We again consider the coefficient of $z$ in the recurrence relations (\ref{eqn.recurrence.relations}). These give
  \begin{align*}
    \delta_n a'_{n-1} + \delta_{n-1} a'_n &= \begin{cases} (c'_{n-1})^K & \text{if $n$ is even,} \\ 0 & \text{if $n$ is odd,} 
    \end{cases}\\
    c'_n c'_{n-1} &= a'_n.
  \end{align*}
  The initial conditions are now $a'_0 = -1$, $c'_0= 1$. As before, we compute 
  \[ a'_{2n+1} = \pm a'_{2n}, \quad c'_{2n+1} = \pm 1, \quad c'_{2n} = \pm a'_{2n}, \quad |a'_{2n}| \leq |a'_{2n-1}|+1.\]
  This proves the claim.
\end{proof}

%\bibliographystyle{plain}
%\bibliography{cox}

\end{document}